%% file: SSVM_submission_arxiv.tex
\documentclass[a4paper]{amsart}

  \usepackage{macros} 
\usepackage{amssymb}
\usepackage{amsmath}
\usepackage{graphicx}
\usepackage{url}

\usepackage{enumerate}
\usepackage{subfig}
\usepackage{color}
\usepackage[usenames,dvipsnames,svgnames,table]{xcolor}
\usepackage{pgfplots} 
\usepackage{tabularx}
\newcolumntype{b}{X}
\newcolumntype{s}{>{\hsize=.65\hsize}X}
\newcolumntype{R}{>{\hsize=.8\hsize}X}
\newcolumntype{t}{>{\hsize=.05\hsize}X}

  \numberwithin{equation}{section}

\begin{document}


%
%

%
%
\title{Asymptotic behaviour of total generalised variation}
%
%
%
\author{Konstantinos Papafitsoros}
\email{kp366@cam.ac.uk}
\author{Tuomo Valkonen}
\email{tjmv3@cam.ac.uk}
\address{Department of Applied Mathematics and Theoretical Physics University of Cambridge, United Kingdom}

\maketitle

\begin{abstract}
The recently introduced   second order total generalised variation functional $\tgvba$ has been a successful regulariser for image processing purposes. Its definition involves two positive parameters $\alpha$ and $\beta$ whose values determine the amount and the quality of the regularisation. In this paper we report on the behaviour of $\tgvba$ in the cases where the parameters $\alpha, \beta$ as well as their ratio $\beta/\alpha$ becomes very large or very small. Among others, we prove that for sufficiently symmetric two dimensional data and large ratio $\beta/\alpha$, $\tgvba$ regularisation coincides with total variation ($\tv$) regularisation.  
\end{abstract}

\section{Introduction}
Parameterisation of variational image processing models has not yet been solved to full satisfaction. Towards the better understanding of such models, we study the behaviour of their solutions  as the parameters change. Within the constraints of these proceedings, we concentrate in particular on the asymptotic behaviour of total generalised variation \cite{TGV}.

In the variational  image reconstruction approach, one typically tries to recover an improved version $u$ of a corrupted image $f$ as a solution of a minimisation problem of the type
\begin{equation}\label{general_min}
\min_{u}\;\mathrm{\Phi}(f,Tu)+\mathrm{\Psi}(u),
\end{equation} 
where $T$ is a linear operator that models the type of corruption. 
Here the term $\mathrm{\Phi}(f,Tu)$ ensures the fidelity of the reconstruction to the initial data. The term $\mathrm{\Psi}(u)$, the \emph{regulariser}, imposes extra regularity on $u$ and it is responsible for the overall quality of the reconstruction. The two terms are balanced by one or more parameters within $\mathrm{\Psi}$. A typical example is $\mathrm{\Psi}(u)=\alpha\tv(u)$, i.e., the total variation of $u$ weighted by a positive parameter $\alpha$ \cite{ChambolleLions,rudin1992nonlinear}.  While total variation regularisation leads to image reconstructions with sharp edges, it also promotes piecewise constant structures leading to the \emph{staircasing effect}. The second order total generalised variation $\tgvba$  \cite{TGV}  resolves that issue by optimally balancing first and second order information in the image data. The $\tgvba$ functional reads
\[\tgvba(u)=\min_{w\in\bd(\om)} \alpha\|Du-w\|_{\cM}+\beta\|\cE w\|_{\cM},\]
where $\|\cdot\|_{\cM}$ is the Radon norm, $\bd(\om)$ is the space of functions of bounded deformation in the domain $\om$, $\cE$ is the symmetrised gradient and $\alpha,\beta>0$.

Since the values of $\alpha$ and $\beta$ determine the amount and the quality of the reconstruction, it is important to understand their role in the regularisation process. In this paper we study the asymptotic behaviour of $\tgvba$ regularised solutions for the extremal  cases, i.e., for large and small values of $\alpha$, $\beta$ and their ratio $\beta/\alpha$. For simplicity we focus on the case where $\mathrm{\Phi}(f,Tu)=\|f-u\|_{L^{2}(\om)}^{2}$ but in most cases, our results can be extended to more general fidelities.

\textbf{Summary of our results:} In Section \ref{sec:small_a_b} we show that as long as at least one of the parameters $\alpha,\beta$ is going to zero then the $\tgvba$ solutions converges to the data $f$. In one dimension we obtain a stronger result, showing in addition that for small values of $\beta$ the solutions are continuous. In Section \ref{sec:large_ba} we focus on the case when the ratio $\beta/\alpha$ is large, proving that in this regime $\tgvba$ is equivalent to $\tv$ modulo ``an affine correction''. In Section \ref{sec:large_a_b} we show that by setting the values of $\alpha$ and $\beta$ large enough we obtain the linear regression of the data as a solution. In Section \ref{sec:large_ba_sym}, we exploit the result of Section \ref{sec:large_ba} and we show that for sufficiently symmetric data and large $\beta/\alpha$, $\tgvba$ is equal to $\alpha\tv$. Our paper is furnished with some numerical experiments in Section \ref{sec:numerics}, which verify our analytical results.

\section{Preliminaries and Notation}

In this section we briefly review the basic theory of functions of bounded variation, properties of $\tv$ and $\tgvba$ and we also fix our notation.

Let $\om$ be an open, bounded domain in $\RR^{d}$. A function $u\in L^{1}(\om)$ is a \emph{function of bounded variation} ($u\in \bv(\om)$) if its distributional derivative $Du$ is represented by an $\RR^{d}$--valued finite Radon measure. The total variation of $u$ is defined as $\tv(u)=\|Du\|_{\cM}$, where $\|\mathcal{T}\|_{\cM}$ denotes the \emph{Radon norm} of an $\RR^{\ell}$--valued distribution $\mathcal{T}$ in $\om$:
\begin{equation}\label{radonnorm}
\|\mathcal{T}\|_{\cM}:= \sup \left \{ \langle \mathcal{T},v \rangle:\; v\in \cC_{c}^{\infty}(\om;\RR^{\ell}),\; \|v\|_{\infty}\le 1\right\},
\end{equation}
and it is equal to the total variation $|Du|(\om)$ of the measure $Du$ when $u\in \bv(\om)$. 
The measure $Du$ can be decomposed into the absolutely continuous  and singular part with respect to the Lebesgue measure $\mathcal{L}^{d}$, 
$Du=D^{a}u+D^{s}u=\nabla u \mathcal{L}^{d}+D^{s}u$, 
where $\nabla u$ is the Radon-Nikod\'ym derivative $D^{a}u/\mathcal{L}^{d}$.
The space $\bv(\om)$ is a Banach space endowed with the norm $\|u\|_{\bv(\om)}=\|u\|_{L^{1}(\om)}+\|Du\|_{\cM}$. We refer the reader to \cite{AmbrosioBV} for a complete account on the functions of bounded variation.

Analogously we define the space of \emph{functions of bounded deformation} $\bd(\om)$ as the set of all the $L^{1}(\om;\RR^{d})$ functions whose symmetrised distributional derivative $\cE u$ is represented by an $\RR^{d\times d}$--valued finite Radon measure  \cite{bd}. Notation-wise one can readily check that $\|\cE u\|_{\cM}=|\cE u|(\om)$. The space $\bv(\om)$ is strictly contained in $\bd(\om)$ for $d>1$ while $\bd(\om)=\bv(\om)$ for one dimensional domains $\om$. We are not going to need much of the theory of  $\bd$ functions apart from the so-called \emph{Sobolev--Korn inequality}. The latter says that if $\om$ has a Lipschitz boundary then  there exists a constant $C_{\bd}>0$ that depends only on $\om$ such that for every $w\in \bd(\om)$ there exists an element $r_{w}\in \mathrm{Ker}\cE$ such that
\begin{equation}\label{Sobolev-Korn}
\|w-r_{w}\|_{L^{1}(\om)}\le C_{\bd} \|\cE w\|_{\cM}.
\end{equation}
Here the kernel of $\cE$ consists of all the functions of the form $r(x)=Ax+b$, where $b\in \RR^{d}$ and $A\in \RR^{d\times d}$ is a skew symmetric matrix.

The \emph{second order total generalised variation} $\tgvba(u)$ of a function $u\in L^{1}(\om)$ is defined as \cite{TGV,BrediesL1,BredValk}
\begin{equation}\label{tgv_w}
\tgvba(u)=\min_{w\in\bd(\om)} \alpha\|Du-w\|_{\cM}+\beta\|\cE w\|_{\cM},
\end{equation}
for $\alpha,\beta>0$. The above definition is usually referred to as the \emph{differentiation cascade} definition of $\tgvba$, see \cite{TGV} for the original formulation.
 It can be shown that $\tgvba$ is a seminorm  and together with  $\|\cdot \|_{L^{1}(\om)}$ form a norm equivalent to $\|\cdot \|_{\bv(\om)}$  \cite{BredValk}, i.e., there exist constants $0<c<C$ that depend only on $\om$ such that for every $u$ with $\tgvba(u)<\infty$
\begin{equation}\label{bv_bgv}
c\|u\|_{\bv(\om)}\le \|u\|_{L^{1}(\om)}+\tgvba(u)\le C\|u\|_{\bv(\om)}.
\end{equation}
Notice that the optimal $w$ in \eqref{tgv_w} is not unique in general.  In fact $w$ is a solution of an $L^{1}$--$\|\cE\|_{\cM}$ problem (not strictly convex). Indeed since $\|Du\|_{\cM}=\|D^{a}u\|_{\cM}+\|D^{s}u\|_{\cM}$, we have:
\begin{align}
w&\in\underset{w\in\bd(\om)}{\operatorname{argmin}}\; \alpha \|Du-w\|_{\mathcal{M}}+\beta\|\cE w\|_{\mathcal{M}}\iff\nonumber\\
w&\in\underset{w\in\bd(\om)}{\operatorname{argmin}}\;\int_{\Omega}|\nabla u-w|~dx+\frac{\beta}{\alpha}\|\cE u\|_{\cM}. \label{L1-E}
\end{align}
In the following sections, we will take specific advantage of the fact that $w$ solves \eqref{L1-E}, a problem which can be seen as an analogous one to $L^{1}$--$\tv$ minimisation.

Let us  finally mention that properties of $\tgvba$ regularisation have been studied in the one dimensional case and when $\mathrm{\Phi}(f,Tu)=\frac{1}{p}\|f-u\|_{L^{p}(\om)}^{p}$ for $p=1$ or $2$, in \cite{BrediesL1,papafitsoros2013study,poschl2013exact}.


\section{Asymptotic behaviour}

\subsection{$\boldsymbol{\beta\to 0}$ while  $\boldsymbol{\alpha}$ is fixed and $\boldsymbol{\alpha\to 0}$ while  $\boldsymbol{\beta}$ is fixed}\label{sec:small_a_b}

In this section we study the limiting behaviour of $\tgvba$ regularisation for small values of $\alpha,\beta$. We first prove that by fixing $\alpha$ or $\beta$ and sending $\beta$ or $\alpha$ to zero respectively, then the regularised $\tgvba$ solution converges to the data $f$. For simplicity we work on the $L^{2}$--$\tgvba$ denoising problem, i.e., $T=Id$,  but the next result can be extended in the more general case e.g. when the fidelity term reads $\frac{1}{p}\|f-Tu\|^{p}$, with $p\ge1$ and $T$ being a bounded, linear operator $T: L^{p}(\om)\to L^{p}(\om)$. For convenience we set
\begin{equation}\label{uab}
(u_{\beta,\alpha},w_{\beta,\alpha})= \underset{\substack{u\in\bv(\om)\\w\in\bd(\om)}}{\operatorname{argmin}}\; \frac{1}{2}\|f-u\|_{L^{2}(\om)}^{2}+\alpha \|Du-w\|_{\mathcal{M}}+\beta\|\cE w\|_{\mathcal{M}}.
\end{equation}

\newtheorem{optimality}{Proposition}
\begin{optimality}\label{lbl:todata}
Let $\om\subseteq \RR^{d}$, open and bounded and $f\in L^{2}(\om)\cap\bv(\om)$. Then
\begin{enumerate}[(i)]
\item Fixing $\alpha>0$ we have that $\|f-u_{\beta,\alpha}\|_{L^{2}(\om)}^{2}\to 0$ as $\beta\to 0$.
\item Fixing $\beta>0$ we have that $\|f-u_{\beta,\alpha}\|_{L^{2}(\om)}^{2}\to 0$ as $\alpha\to 0$.
\end{enumerate}
\end{optimality}

\begin{proof}
$(i)$ Let $\epsilon>0$ and $\{\rho_{\delta}\}_{\delta>0}$ be a standard family of mollifiers, i.e., $\rho_{\delta}(x)=\delta^{-d}\rho(x/\delta)$, where $\rho\in \mathcal{C}_{c}^{\infty}(\RR^{d})$, and set $f_{\delta}:=\rho_{\delta}\ast f$.
Because $(u_{\beta,\alpha},w_{\beta,\alpha})$ is an optimal pair in \eqref{uab} by setting $u=f_{\delta}$ and $w=\nabla f_{\delta}$ we have the following estimates, for some constant $C>0$
\begin{align*}
\frac{1}{2}\|f-u_{\beta,\alpha}\|_{L^{2}(\om)}^{2}& \le \frac{1}{2}\|f-u_{\beta,\alpha}\|_{L^{2}(\om)}^{2}+\alpha \|Du_{\beta,\alpha}-w_{\beta,\alpha}\|_{\mathcal{M}}+\beta\|\cE w_{\beta,\alpha}\|_{\mathcal{M}}\\
									& \le \frac{1}{2}\|f-f_{\delta}\|_{L^{2}(\om)}^{2}+\beta\|\cE (\nabla f_{\delta})\|_{\cM}\\
									& \le  \frac{1}{2}\|f-f_{\delta}\|_{L^{2}(\om)}^{2}+\beta\|\cE \rho_{\delta} \ast Df\|_{\cM}\\
									&\le  \frac{1}{2}\|f-f_{\delta}\|_{L^{2}(\om)}^{2}+\beta\frac{1}{\delta} \|Df\|_{\cM}.
\end{align*}
We set $\delta$ small enough such that $\|f-f_{\delta}\|_{L^{2}(\om)}^{2}\le \epsilon/2$. By choosing $\beta< \delta \epsilon/2\|Df\|_{\cM}$, the result follows. 

\noindent
$(ii)$ The proof is very similar to the $(i)$ case, by setting $u=f_{\delta}$ and $w=0$, instead.
\end{proof}

\noindent
\textbf{Remark}: Of course in both $(i)$--$(ii)$ cases of Proposition \ref{lbl:todata}, we also get \\$\|Du_{\beta,\alpha}-w_{\beta,\alpha}\|_{\cM}\to 0$ as well as $\|\cE w_{\beta,\alpha}\|_{\cM}\to 0$ as $\beta\to 0$ or $\alpha\to 0$. 

Another interesting behaviour occurs when $\beta\to 0$.  In \cite{valkonen2014jump2}, it is proved that for an arbitrary dimension and  a fixed $\alpha>0$ we have 
\[\|D^{s}u_{\beta,\alpha}\|_{\cM}\to 0,\quad \text{as }\beta\to 0.\]
However it turns out that in dimension one we are able to prove something stronger, provided the data are bounded:
\newtheorem{Dsu0}[optimality]{Proposition}
\begin{Dsu0}\label{lbl:Dsu0}
Let $\om=(a,b)\subseteq \RR$, $f\in L^{\infty}(\om)\cap\bv(\om)$ and $\alpha>0$. Then there exists a threshold $\beta^{\ast}>0$ such that for every $\beta<\beta^{\ast}$ we have that
\[\|D^{s}u_{\beta,\alpha}\|_{\cM}=0 \quad \text{and}\quad w_{\beta,\alpha}=\nabla u_{\beta,\alpha}.\]
In particular this means that for $\beta<\beta^{\ast}$
\begin{equation}\label{uab_b0}
u_{\beta,\alpha}= \underset{\substack{u\in\bv(\om)}}{\operatorname{argmin}}\; \frac{1}{2}\|f-u\|_{L^{2}(\om)}^{2}+\beta\|D^{2} u\|_{\mathcal{M}}.
\end{equation}
\end{Dsu0}

\begin{proof}
From the optimality conditions derived in \cite{papafitsoros2013study}, we have that $(u_{\beta,\alpha},w_{\beta,\alpha})$ solve \eqref{uab} if and only if there exists a dual variable  $v\in H_{0}^{2}(\Omega)$ such that
%
\[v''=f-u_{\beta,\alpha}\;\;(C_{f}),\;\; -v'\in \alpha\mathrm{Sgn}(Du_{\beta,\alpha}-w_{\beta,\alpha}) \;\;(C_{\alpha}),\;\;  v\in\beta\mathrm{Sgn}(Dw_{\beta,\alpha})\;\; (C_{\beta}),\]
where for a finite Radon measure $\mu$ we define
\[\sgn (\mu):= \left \{v\in L^{\infty}(\om)\cap L^{\infty}(\om,|\mu|):\; \|v\|_{\infty}\le 1,\; v=\frac{\mu}{|\mu|},\; |\mu|-a.e. \right \}.\]
 Note also that there exists a constant $C$ depending only on $\om$ such that the following interpolation inequality holds \cite[Section 5.10, ex. 9]{EvansPDEs}
\begin{equation}\label{interp}
\|Dv\|_{L^{2}(\om)}\le C \|v\|_{L^{2}(\om)}^{1/2}\|D^{2}v\|_{L^{2}(\om)}^{1/2},\quad \text{for all } v\in H_{0}^{2}(\om).
\end{equation}
Observe first that (denoting this dual function $v$ by $v_{\beta,\alpha}$)
\begin{equation}\label{vprimeL2}
\|Dv_{\beta,\alpha}\|_{L^{2}(\om)}\to 0\quad \text{as } \beta\to 0.
\end{equation}
 Indeed, from Proposition \ref{lbl:todata} and condition $(C_{f})$ we have that $\|D^{2}v_{\beta,\alpha}\|_{L^{2}(\om)}\to 0$ while from condition $(C_{\beta})$ we have that $\|v_{\beta,\alpha}\|_{\infty}\to 0$ and thus $\|v_{\beta,\alpha}\|_{L^{2}(\om)}\to 0$ as $\beta\to 0$. Then we just apply the estimate \eqref{interp}.

From the fact that we are in dimension one and from \eqref{bv_bgv} we have for a generic constant $C$
\begin{align*}
\|u_{\beta,\alpha}\|_{L^{\infty}(\om)} &\le C\|u_{\beta,\alpha}\|_{\bv(\om)}\le C(\|u_{\beta,\alpha}\|_{L^{2}(\om)}+\tgvba(u_{\beta,\alpha}))\\
						       & \le C(\|f-u_{\beta,\alpha}\|_{L^{2}(\om)}+\|f\|_{L^{2}(\om)}+\tgvba(u_{\beta,\alpha}))\\
						       &\le C(\|f\|_{L^{2}(\om)},\tgvba(f)):=M,
\end{align*}
which in combination with condition $(C_{f})$ and the fact that $f\in L^{\infty}(\om)$ implies that $\|D^{2}v_{\beta,\alpha}\|_{\infty}\le M$. Thus from the Arzel\`a-Ascoli theorem we get the existence of a sequence $\beta_{n}\to 0$ and a continuous function $\tilde{v}$ such that $v_{\beta_{n},\alpha}\to \tilde{v}$ uniformly. We immediately deduce using \eqref{vprimeL2} that $v_{\beta,\alpha}\to 0$ uniformly as $\beta\to 0$. But then condition $(C_{\alpha})$ implies that there must exist a $\beta_{0}$ such that for every $\beta<\beta_{0}$ we have $Du_{\beta,\alpha}=w_{\beta,\alpha}$, as measures, i.e., $D^{s}u_{\beta,\alpha}=0$ and $w_{\beta,\alpha}=\nabla u_{\beta,\alpha}$ since otherwise there would exist a point $x_{\beta_{n},\alpha}\in (a,b)$ with $Dv_{\beta_{n},\alpha}(x_{\beta_{n},\alpha})=\alpha$ for a sequence $(\beta_{n})_{n\in \NN}$ converging to 0, a contradiction.
\end{proof}

\noindent
\textbf{Remark}: 
We believe that the above proof
sets the basis for an analogue proof in higher dimensions even though admittedly this is a hard task.
 That would require an interpolation inequality for $v$, $\mathrm{div}v$ and $\mathrm{div}^{2}v$ analogous to \eqref{interp},  as well as a proof that the $\tgvba$ regularised solution remains bounded, for bounded data $f$.

\subsection{Large ratio $\boldsymbol{\beta/\alpha}$}\label{sec:large_ba}

Recall from \eqref{L1-E} that the optimal $w$ is a solution to a $L^{1}$--$\|\cE\|_{\cM}$ type of problem. This motivates us to study some particular properties of the general form of such a problem:
\begin{equation}\label{general-L1-E}
\min_{w\in \bd(\om)} \|g-w\|_{L^{1}(\om;\RR^{d})}+\lambda \|\cE w\|_{\cM},\quad g\in L^{1}(\om;\RR^{d}),\;\lambda>0.
\end{equation}
The next theorem states that if the parameter $\lambda$ is larger than a certain threshold (depending only on $\om$) then a solution $w$ of \eqref{general-L1-E} will belong to $\mathrm{Ker}\cE$. This is analogous to the $L^{1}$--$\tv$ problem \cite{chanL1,duvalL1}, where there for large enough value of the parameter $\lambda$, the solution is constant, i.e., belongs to the kernel of $\tv$.

\newtheorem{L1-E-threshold}[optimality]{Proposition}
\begin{L1-E-threshold}\label{lbl:L1-E-threshold}
Let $\om\subseteq \RR^{d}$ be an open, bounded set with Lipschitz boundary, $g\in L^{1}(\om;\RR^{d})$ and $C_{\bd}$ the constant that appears in the Sobolev--Korn inequality \eqref{Sobolev-Korn}. Then if $\lambda> C_{\bd}$ and $w_{\lambda}$ is a solution of \eqref{general-L1-E} with parameter $\lambda$, then
\begin{equation}\label{L1-E-threshold:main}
w_{\lambda}=m_{\cE}(g):=\underset{w\in\mathrm{Ker}\cE}{\operatorname{argmin}}\;  \|g-w\|_{L^{1}(\om;\RR^{d})}.
\end{equation}
\end{L1-E-threshold}

\begin{proof}
Since $w_{\lambda}$ is a solution of \eqref{general-L1-E}, it is easy to check that if $r_{w_{\lambda}}$ is the element of $\mathrm{Ker}\cE$ that corresponds to $w_{\lambda}$ in the Sobolev-Korn inequality then, $W_{\lambda}:= w_{\lambda}-r_{w_{\lambda}}$ solves the following problem:
\begin{equation}\label{L1-E-W}
\min_{w\in\bd(\om)} \|(g-w_{\lambda})-w\|_{L^{1}(\om;\RR^{d})}+\lambda\|\cE w\|_{\cM}.
\end{equation}
Indeed, we have for an arbitrary $w\in\mathrm{BD}(\om)$
\begin{align*}
&\|(g-r_{w_{\lambda}})-W_{\lambda}\|_{L^{1}(\om;\RR^{d})}+\lambda\|\mathcal{E}W_{\lambda}\|_{\cM}\le \|(g-r_{w_{\lambda}})-w\|_{L^{1}(\om;\RR^{d})}+\lambda\|\mathcal{E}w\|_{\cM}, \\
& \hspace{6cm}\iff\\
&\|g-w_{\lambda}\|_{L^{1}(\om;\RR^{d})}+\lambda\|\mathcal{E}w_{\lambda}\|_{\cM}\le \|g-(w+r_{w_{\lambda}})\|_{L^{1}(\om;\RR^{d})}+\lambda\|\mathcal{E}(w+r_{w_{\lambda}})\|_{\cM},
\end{align*}
with the latter being true since
\[\|g-w_{\lambda}\|_{L^{1}(\om;\RR^{d})}+\lambda\|\mathcal{E}w_{\lambda}\|_{\cM}\le \|g-w\|_{L^{1}(\om;\RR^{d})}+\lambda\|\mathcal{E}w\|_{\cM},\quad \forall w\in\mathrm{BD}(\om).\]
Since $W_{\lambda}$ solves \eqref{L1-E-W}, setting $G_{\lambda}:=g-w_{\lambda}$ we have that
\[\|G_{\lambda}-W_{\lambda}\|_{L^{1}(\om;\RR^{d})}+\lambda \|\cE W_{\lambda}\|_{\cM}\le \|G_{\lambda}\|_{L^{1}(\om;\RR^{d})},\]
and using the Sobolev--Korn inequality $\|W_{\lambda}\|_{L^{1}(\om;\RR^{d})}\le C_{\bd}\|\cE W_{\lambda}\|_{\cM}$ we have
\begin{equation}\label{W_final}
\|G_{\lambda}-W_{\lambda}\|_{L^{1}(\om;\RR^{d})}+\frac{\lambda}{C_{\bd}} \| W_{\lambda}\|_{L^{1}(\om;\RR^{d})}\le \|G_{\lambda}\|_{L^{1}(\om;\RR^{d})}.
\end{equation}
A simple application of the triangle inequality in \eqref{W_final} yields that if $\lambda>C_{\bd}$, then we must have $W_{\lambda}=0$, i.e., $w_{\lambda}=r_{w_{\lambda}}$ from which \eqref{L1-E-threshold:main} straightforwardly follows.
\end{proof}

The notation $m_{\cE}(g)$ can be interpreted as the median of $g$ with respect to $\mathrm{Ker}\cE$. If $d=1$, then this is nothing else than the usual median since in that case $\mathrm{Ker}\cE$ consists of all the constant functions.
The following corollary follows immediately from \eqref{L1-E} and Proposition \eqref{lbl:L1-E-threshold}. It says that for large $\beta/\alpha$, $\tgvba$ is almost equivalent to $\tv$ up to an ``affine correction''.

\newtheorem{TGV_largeba}[optimality]{Corollary}
\begin{TGV_largeba}\label{lbl:TGV_largeba}
Let $\om\subseteq \RR^{d}$ be an open, bounded set with Lipschitz boundary and let $\alpha,\beta>0$ such that $\beta/\alpha> C_{\bd}$. Then for every $u\in \bv(\om)$
\[\tgvba(u)=\alpha\|Du-m_{\cE}(\nabla u)\|_{\cM}.\]

\end{TGV_largeba}

\subsection{Thresholds for regression}\label{sec:large_a_b}

In this section we show that there exist some thresholds for $\alpha$ and $\beta$ above which the solution to the $L^{2}$--$\tgvba$ regularisation problem  is the $L^{2}$-linear regression of the data $f$, denoted by $f^{\star}$:
\[f^{\star}:=\underset{\phi \text{ affine}}{\operatorname{argmin}}\;\|f-\phi\|_{L^{2}(\om)}^{2}.\]
We are going to need the following proposition proved in \cite{BredValk}:
\newtheorem{poincare_tgv_thm}[optimality]{Proposition}
\begin{poincare_tgv_thm}[{\cite[Proposition~4.1]{BredValk}}]
Let $\om\subseteq \RR^{d}$ be a bounded, open set with Lipschitz boundary. Then for every $1\le p\le d/(d-1)$, there exists a constant $C_{\mathrm{BGV}}(\beta/\alpha)>0$, that depends only on $\om$, $p$ and the ratio $\beta/\alpha$ such that
\begin{equation}\label{poincare_tgv}
\|u-u^{\star}\|_{L^{p}(\om)}\le C_{\mathrm{BGV}}(\beta/\alpha)\tgv_{\beta/\alpha,1}^{2}(u).
\end{equation}
\end{poincare_tgv_thm}

In the next proposition we show the existence of these regression thresholds for $d=2$ and also for $d>2$ under the extra assumption that the $L^{p}$ norm of the data $f$ controls the $L^{p}$ norm of the solution for some $p\in[d,\infty]$.
\newtheorem{regression_L2}[optimality]{Proposition}
\begin{regression_L2}\label{lbl:regression_L2}
Let $\om\subseteq \RR^{d}$ be a bounded, open set with Lipschitz boundary. Suppose that either 
\begin{enumerate}[(i)]
\item $d=2$ and $f\in\bv(\om)$ or
\item $d>2$, $f\in L^{\infty}(\om)\cap\bv(\om)$ and there exists a constant $C>0$ depending only on the domain and $p\in[d,\infty]$ such that $\|u\|_{L^{p}(\om)}\le C\|f\|_{L^{p}(\om)}$ for $u$ solution to \eqref{uab}, 
\end{enumerate}
then there exist $\alpha^{\star}, \beta^{\star}>0$ such that whenever $\alpha>\alpha^{\star}$, $\beta>\beta^{\star}$ then the solution to the $L^{2}$--$\tgvba$ regularisation problem is equal to $f^{\star}$.
\end{regression_L2}
\begin{proof}
Suppose initially that $d=2$ and $f\in\bv(\om)$. Then using the H\"older inequality along with \eqref{poincare_tgv} and the fact that any function $u\in\bv(\om)$
that solves the $L^{2}$--$\tgvba$ problem has a $L^{2}$ norm bounded by a constant $C$ depending only on $f$ and not on $\alpha,\beta$
\begin{align}
\frac{1}{2}\|f-f^{\star}\|_{L^{2}(\om)}^{2} &=\min_{\phi \text{ affine}}\frac{1}{2}\|f-\phi\|_{L^{2}(\om)}^{2}\le \frac{1}{2}\|f-u^{\star}\|_{L^{2}(\om)}^{2} \nonumber\\
							    &= \frac{1}{2}\|f-u\|_{L^{2}(\om)}^{2}+\frac{1}{2}\|u-u^{\star}\|_{L^{2}(\om)}^{2}+\int_{\om}(f-u)(u-u^{\star})dx \nonumber\\
							    &\le \frac{1}{2}\|f-u\|_{L^{2}(\om)}^{2} +C(f)\|u-u^{\star}\|_{L^{2}(\om)}\label{holder}\\
							    &\le\frac{1}{2}\|f-u\|_{L^{2}(\om)}^{2} +C(f)C_{\mathrm{BGV}}(\beta/\alpha)\mathrm{TGV}_{\beta/\alpha,1}^{2}(u). \nonumber					    
\end{align}
Setting $\alpha^{\star}=C(f)C_{\mathrm{BGV}}(1)$ and $\beta^{\star}=\alpha^{\star}$ we have that if $\alpha>\alpha^{\star}$ and $\beta>\beta^{\star}$ we can further estimate
\begin{align*}
\frac{1}{2}\|f-f^{\star}\|_{L^{2}(\om)}^{2}&\le\frac{1}{2}\|f-u\|_{L^{2}(\om)}^{2} +C(f)C_{\mathrm{BGV}}(1)\mathrm{TGV}_{\beta^{\star}/\alpha^{\star},1}^{2}(u)\\
							  &\le\frac{1}{2}\|f-u\|_{L^{2}(\om)}^{2} +\alpha^{\star}\mathrm{TGV}_{\beta^{\star}/\alpha^{\star},1}^{2}(u)\\
							  & \le\frac{1}{2}\|f-u\|_{L^{2}(\om)}^{2} +\mathrm{TGV}_{\beta^{\star},\alpha^{\star}}^{2}(u)\\
							  &\ \le\frac{1}{2}\|f-u\|_{L^{2}(\om)}^{2} +\tgvba(u).						
\end{align*}
The proof goes through for the case $(ii)$ as well, where the only difference is that H\"older inequality in \eqref{holder} gives two terms $\|u-u^{\star}\|_{L^{p}(\om)}\|u-u^{\star}\|_{L^{p^{\ast}}(\om)}$ and $\|f-u\|_{L^{p}(\om)}\|u-u^{\star}\|_{L^{p^{\ast}}(\om)}$, where $p^{\ast}=p/(p-1)$ and $\infty^{\ast}:=1$. These terms can be further bounded using inequality \eqref{poincare_tgv} (note that $p^{\ast}\le d/(d-1)$) and the fact that $\|u\|_{L^{p}(\om)}\le C\|f\|_{L^{p}(\om)}$.
\end{proof}

More explicit regression thresholds have been given in \cite{papafitsoros2013study} both for general and specific one dimensional data $f$. Let us point out that the condition $\|u\|_{L^{p}(\om)}\le C\|f\|_{L^{p}(\om)}$ and in particular $\|u\|_{\infty}\le C\|f\|_{\infty}$ (which can be derived easily for $\tv$ regularisation with $C=1$), as natural as it may seems, it cannot be shown easily. However, if proved, it will also have positive implications as far as the inclusion of the jump set of the solution to the jump set of the data is concerned, see \cite{valkonen2014jump2}.

\subsection{Equivalence to $\boldsymbol{\tv}$ for large ratio $\boldsymbol{\beta/\alpha}$ and sufficiently symmetric data}\label{sec:large_ba_sym}

In Corollary \ref{lbl:TGV_largeba} we obtained a more precise characterisation of $\tgvba$ for large values of the ratio $\beta/\alpha$. In this section we show that at least for symmetric enough data $f$, $\tgvba$ regularisation is actually equivalent to $\alpha\tv$ regularisation. For the sake of the simplicity of the analysis we assume here that $\om$ is a two dimensional domain, i.e., $\om\subseteq \RR^{2}$. We will also need some symmetry for $\om$, for the time being let $\om$ be a square centered at the origin. We shall prove the following theorem.

\newtheorem{tgv_tv_thm}[optimality]{Theorem}
\begin{tgv_tv_thm}\label{lbl:tgv_tv_thm}
Suppose that $\om\subseteq\RR^{2}$ is a bounded, open square, centred at the origin and let $f\in\bv(\om)$ satisfy the following symmetry properties:
\begin{enumerate}[(i)]
\item $f$ is symmetric with respect to both axes, i.e.,
\[f(x_{1},x_{2})=f(-x_{1},x_{2}),\quad f(x_{1},x_{2})=f(x_{1},-x_{2}),\quad \text{for a.e. }(x_{1},x_{2})\in\om.\]
\item $f$ is invariant under $\pi/2$ rotations, i.e., $f(O_{\pi/2}x)=f(x)$, where $O_{\pi/2}$ denotes counterclockwise rotation by $\pi/2$ degrees.
\end{enumerate}
Then if $\beta/\alpha> C_{\bd}$, the problems 
\[\min_{u\in\bv(\om)} \frac{1}{p}\int_{\om}|f-u|^{p}dx+\tgvba(u)\quad \text{and}\quad\min_{u\in\bv(\om)} \frac{1}{p}\int_{\om}|f-u|^{p}dx+\alpha\tv(u)\]
for $p\ge 1$ are equivalent.
\end{tgv_tv_thm}

\newtheorem{tgv_tv_rem}[optimality]{Remark}
\begin{tgv_tv_rem}\label{lbl:tgv_tv_rem}
\emph{The proof of Theorem \ref{lbl:tgv_tv_thm} is essentially based on the fact that the symmetry of the data $f$ is inherited to the solution $u$ and thus to $\nabla u$. In that case we can show that $m_{\cE}(\nabla u)=0$ something that shows the equivalence of $\tgvba$ and $\alpha\tv$. Other symmetric domains, e.g. circles, rectangles, together with appropriate symmetry conditions for $f$ can also guarantee that $\nabla u$ has the desired symmetry properties as well. The same holds for any fidelities $\Phi(f,Tu)$ that ensure that the symmetry of $f$ is passed to $u$.}
\end{tgv_tv_rem}
Let us also mention that abusing the notation a bit, by $m_{\cE}(\nabla u)=0$ we mean that zero is a solution of the problem \eqref{L1-E-threshold:main} with $g=\nabla u$.

\begin{proof}[of Theorem \ref{lbl:tgv_tv_thm}]
Since $\beta/\alpha>C_{\bd}$, from Corollary \ref{lbl:TGV_largeba} we have that the $\tgvba$ regularisation problem is equivalent to 
\begin{equation}\label{Lp_tgv_largeba}
\min_{u\in\bv(\om)} \frac{1}{p}\int_{\om}|f-u|^{p}dx+\alpha\|Du-m_{\cE}(\nabla u)\|_{\cM}.
\end{equation}
Thus it suffices to show that $m_{\cE}(\nabla u)=0$. Since $f$ satisfies the symmetry properties $(i$)--$(ii)$, from the rotational invariance of $\tgvba$ \cite{TGV} we have that the same conditions hold for the $\tgvba$ regularised solution $u$. This also means that $\nabla u=(\nabla_{1}u,\nabla_{2}u)$ has the following properties for almost all $x=(x_{1},x_{2})\in\om$:
\begin{align}
\nabla_{1}u(x_{1},x_{2})&=\nabla_{1}u(x_{1},-x_{2}),\quad &\nabla_{2}u(x_{1},x_{2})&=\nabla_{2} u(-x_{1},x_{2}),\label{symmetry_nablau1}\\
\nabla u(x)&=-\nabla u(-x), &\nabla_{1}u(O_{\pi/2}x)&=\nabla_{2} u(x). \label{symmetry_nablau2}
\end{align}
Recalling that
\[m_{\cE}(\nabla u)=\underset{w\in\mathrm{Ker}\cE}{\operatorname{argmin}}\;  \|\nabla u-w\|_{L^{1}(\om;\RR^{d})},\]
and that $m_{\cE}(\nabla u)$ has the form $Ax+b$ it is easy to check, see the following lemma, that $m_{\cE}(\nabla u)=0$.
\end{proof}

\newtheorem{symmetry_lemma}[optimality]{Lemma}
\begin{symmetry_lemma}\label{lbl:symmetry_lemma}
Let $\om$ be a square centred at the origin and suppose that $g=(g_{1},g_{2})\in L^{1}(\om;\RR^{2})$ satisfies the symmetry properties
\begin{align}
g(x)=-g(-x),\quad g_{1}(x_{1},x_{2})=g(x_{1},-x_{2}),\quad& g_{2}(x_{1},x_{2})=g_{2}(-x_{1},x_{2}), \label{symmetry_g1}\\
&g_{1}(O_{\pi/2}x)=g_{2}(x),\label{symmetry_g1}
\end{align} 
for almost every $x=(x_{1},x_{2})\in \om$. Then the minimisation problem
\begin{equation}\label{min_g}
\min_{w\in \mathrm{Ker}\cE} \|g-w\|_{L^{1}(\om;\RR^{2})},
\end{equation}
admits $w=0$ as a solution.
\end{symmetry_lemma}
\begin{proof}
Recalling that $\mathrm{Ker}\cE$ consists of all the functions of the form $r(x)=Ax+b$ with $A$ being a skew symmetric function, we have that the minimisation \eqref{min_g} is equivalent to
\begin{equation}\label{min_g_equi1}
\min_{A,b} \int_{\om} |g(x)-Ax-b|dx,
\end{equation}
with corresponding optimality conditions
\[\int_{\om}\left \langle \frac{g(x)-AO_{\pi/2}x-b}{|g(x)-AO_{\pi/2}x-b|}, O_{\pi/2}x \right\rangle dx=0,\quad\text{with }O_{\pi/2}=
\left( \begin{array}{cc}
0 & -1 \\
1 & 0  \end{array} \right).
\]
Using the equalities $g_{2}(x_{1},x_{2})=g_{1}(-x_{2},x_{1})$ and $g_{1}(-x_{2},x_{1})=g_{1}(-x_{2},-x_{1})=-g_{1}(x_{1},x_{2})$ we have that $A=0$, $b=0$ solve \eqref{min_g_equi1} if 
\begin{align*}
&\int_{\om} \left \langle\frac{O_{\pi/2}g(x)}{|O_{\pi/2}g(x)|},x  \right\rangle dx=0\iff
\int \frac{x_{2}g_{1}(x_{1},x_{2})-x_{1}g_{2}(x_{1},x_{2})}{|g(x)|}dx=0\iff\\
&\int_{\om}\frac{x_{2}g_{1}(x_{1},x_{2})-x_{1}g_{1}(-x_{2},x_{1})}{\sqrt{g_{1}(x_{1},x_{2})^{2}+g_{1}(-x_{2},x_{1})^{2}}}dx=0\iff
\int_{\om}(x_{1}+x_{2})\frac{g_{1}(x_{1},x_{2})}{|g_{1}(x_{1},x_{2})|}dx=0,
\end{align*}
with last equality being true since $-g(-x)=g(x)$.
\end{proof}

\subsection{Numerical experiments}\label{sec:numerics}

\begin{figure}[h!]
\begin{center}
\subfloat[Symmetric  data]{
\includegraphics[width=0.21\textwidth]{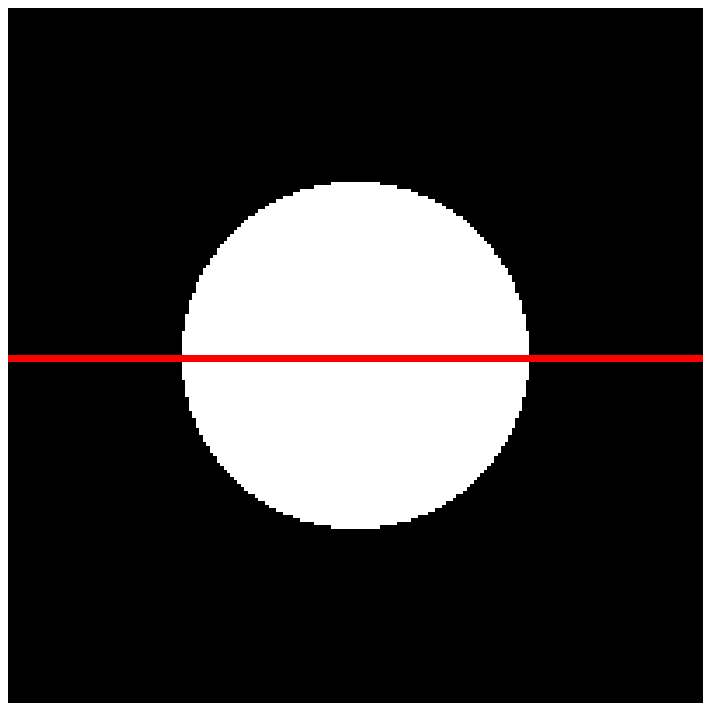}
}
\subfloat[$\tv$ solution,\newline $\alpha=10$]{
\includegraphics[width=0.21\textwidth]{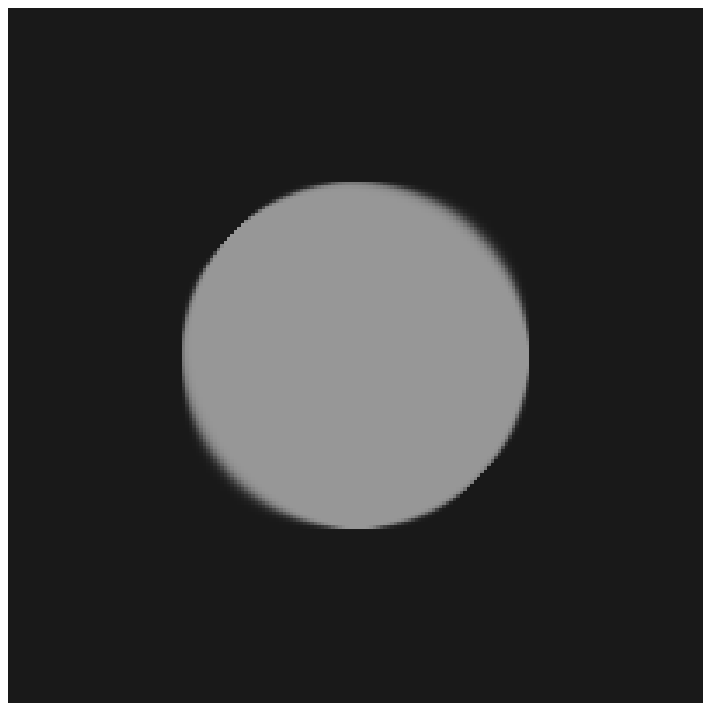}
}
\subfloat[$\tgv$ solution,\newline $\alpha=10$, $\beta=10^{6}$]{
\includegraphics[width=0.21\textwidth]{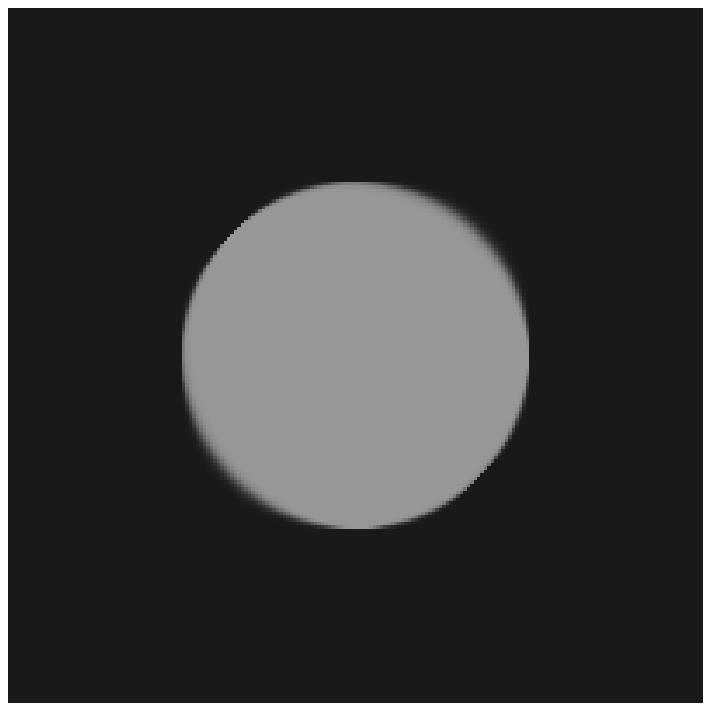}
}
\subfloat[$\tgv$ solution,\newline $\alpha=10$, $\beta=200$]{
\includegraphics[width=0.21\textwidth]{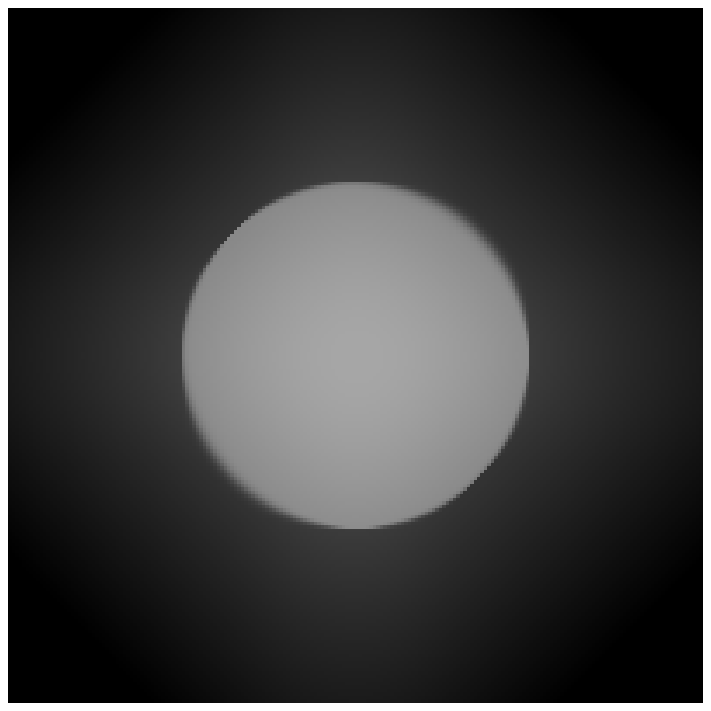}
}

\subfloat[Non-symmetric\newline data]{
\includegraphics[width=0.21\textwidth]{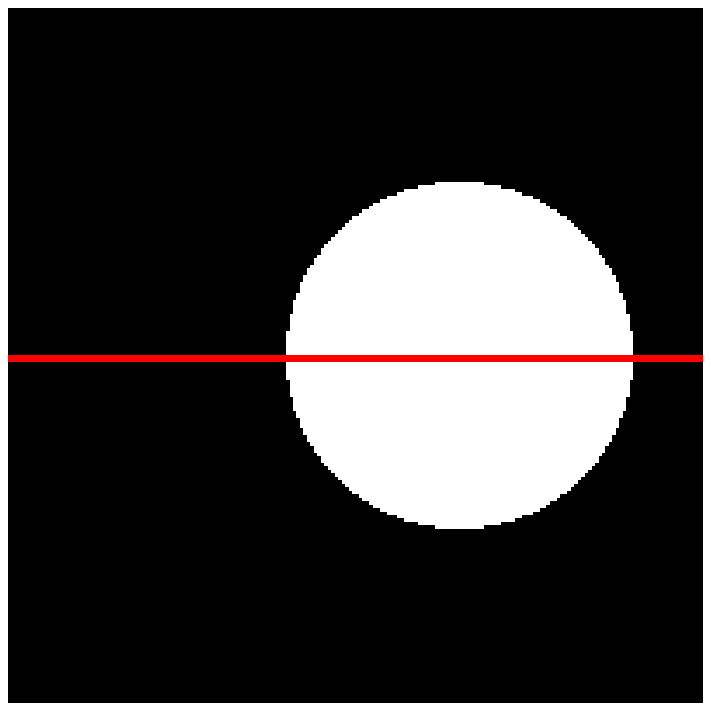}
}
\subfloat[$\tv$ solution,\newline $\alpha=10$]{
\includegraphics[width=0.21\textwidth]{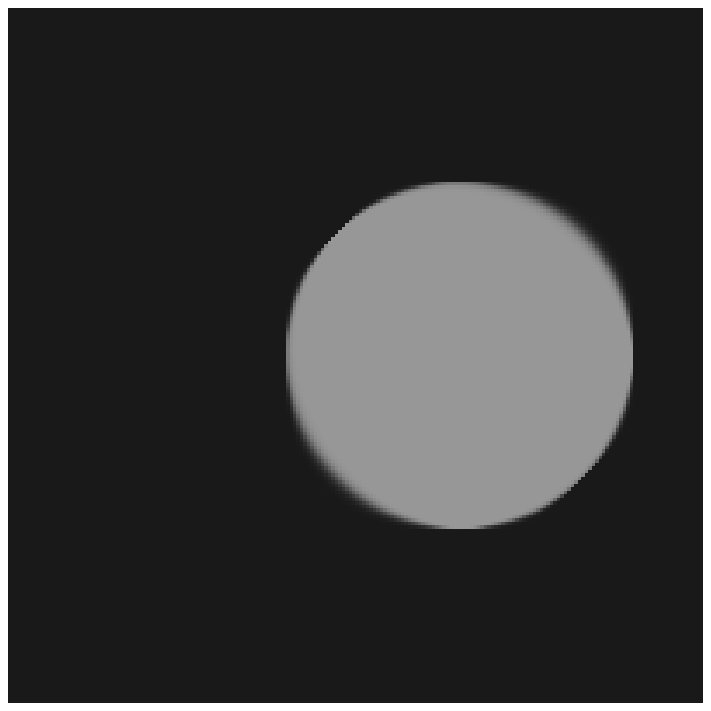}
}
\subfloat[$\tgv$ solution,\newline $\alpha=10$, $\beta=10^{6}$]{
\includegraphics[width=0.21\textwidth]{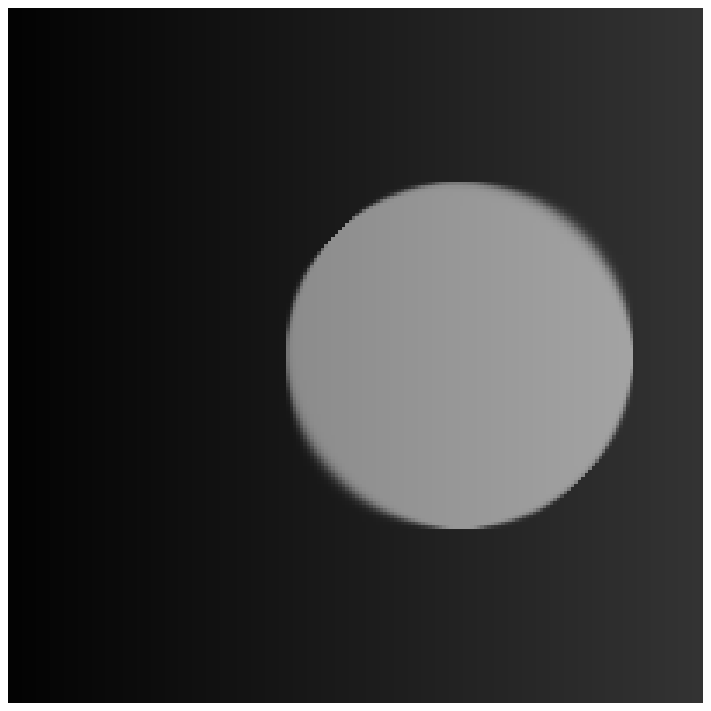}
}
\subfloat[$\tgv$ solution,\newline $\alpha=10$, $\beta=10^{6}$]{
\includegraphics[width=0.21\textwidth]{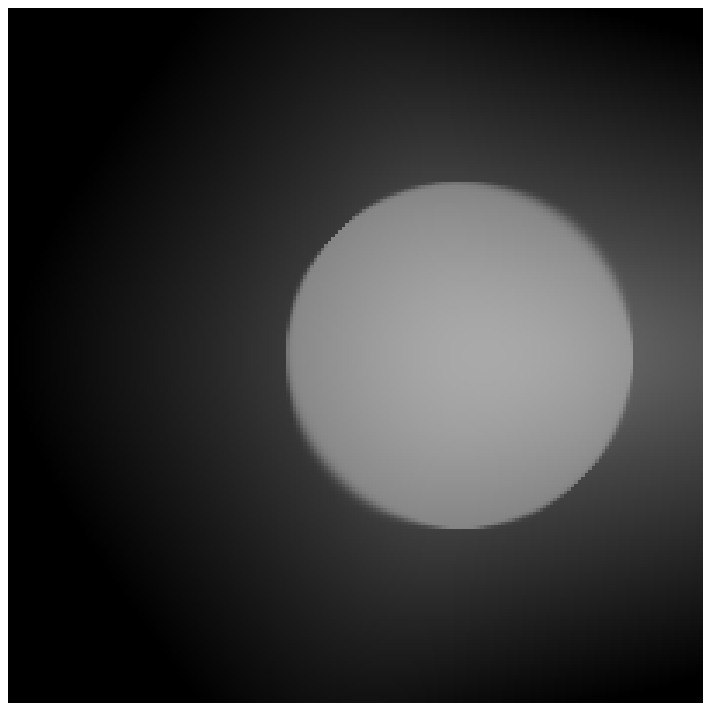}
}

\subfloat[Corresponding middle row slices for \newline symmetric data]{
\includegraphics[width=0.48\textwidth]{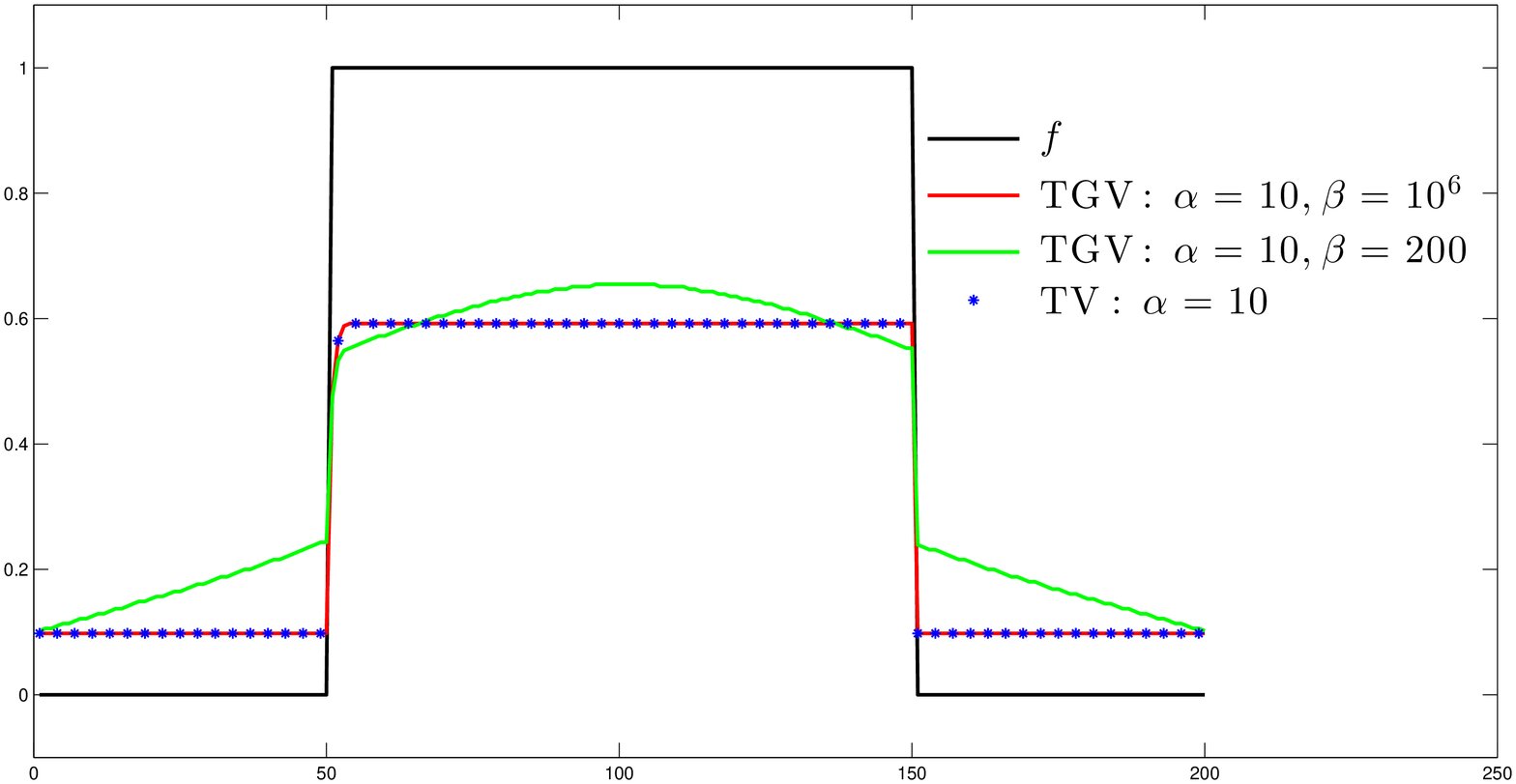}
}
\subfloat[Corresponding middle row slices for non-symmetric data]{
\includegraphics[width=0.48\textwidth]{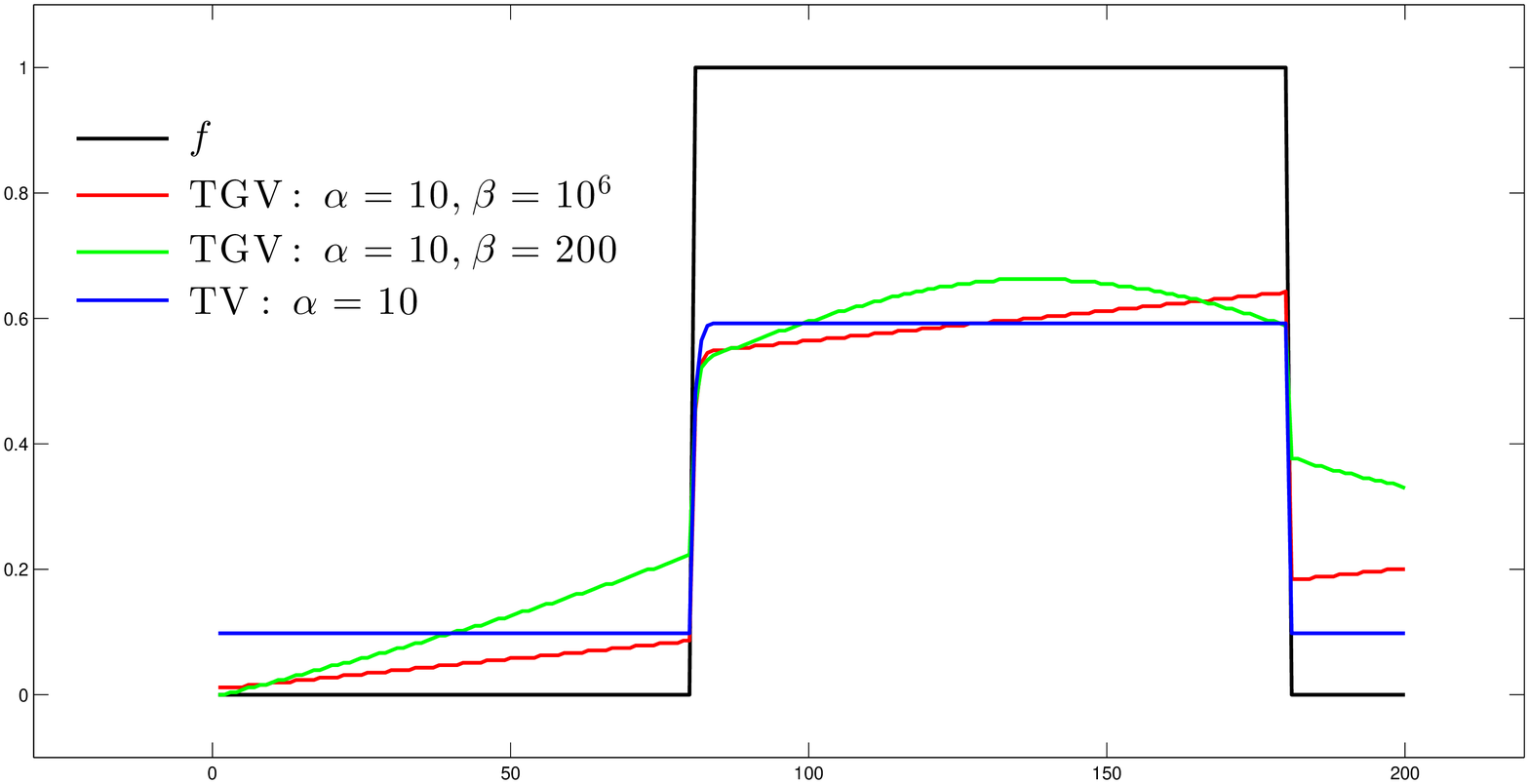}
}

\caption{Illustration of the two dimensional $\alpha\tv$ and $\tgvba$ equivalence for symmetric data when $\beta/\alpha$ is large enough.  Notice that the equivalence does not hold once the symmetry is broken.}
\label{square_2d_1}
\end{center}
\end{figure}

In this section we verify some of our results using numerical experiments. 
 In Figure \ref{square_2d_1} we confirm Theorem \ref{lbl:tgv_tv_thm}. There, we apply $\alpha\tv$ and $\tgvba$ denoising with $L^{2}$ fidelity, to a characteristic function of a disk centred at the middle of the domain, Figure \ref{square_2d_1}(a) and away from it, Figure \ref{square_2d_1}(e). Notice that the symmetry properties of Theorem \ref{lbl:tgv_tv_thm} are satisfied for the first case. There, we observe that by choosing the ratio $\beta/\alpha$  large enough, $\tgvba$ and $\alpha\tv$ solutions coincide, Figures \ref{square_2d_1}(b) and  \ref{square_2d_1}(c) . However, they do not coincide for small ratio $\beta/\alpha$, Figure \ref{square_2d_1}(d), see also the middle row slices in Figure \ref{square_2d_1}(i). In this case $\tgvba$ produces a piecewise smooth result in comparison to the piecewise constant one of $\alpha\tv$.
 Note that when the symmetry is broken, $\alpha\tv$ and $\tgvba$ solutions do not coincide even for large ratio $\beta/\alpha$, Figures \ref{square_2d_1}(g), \ref{square_2d_1}(h) and \ref{square_2d_1}(j). 
 
  Figure \ref{squares} depicts another example of an image that satisfies the symmetry properties of Theorem \ref{lbl:tgv_tv_thm}.  The $\alpha\tv$ solution coincides with the $\tgvba$ one for large ratio  $\beta/\alpha$, Figures \ref{squares}(b) and \ref{squares}(c), but not for small ratio, Figure \ref{squares}(d).

\begin{figure}[h!]
\begin{center}
\subfloat[Original image]{
\includegraphics[width=0.21\textwidth]{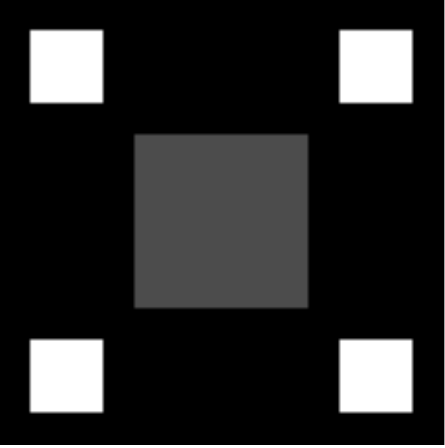}
}
\subfloat[$\tv$ solution,\newline $\alpha=1$]{
\includegraphics[width=0.21\textwidth]{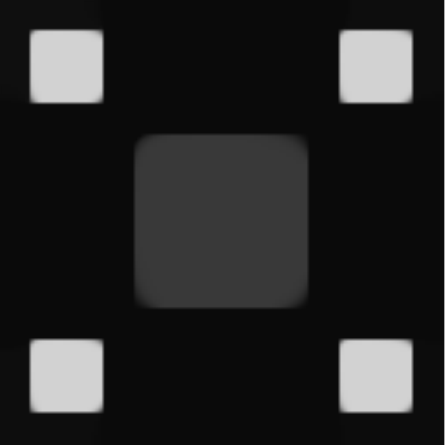}
}
\subfloat[$\tgv$ solution,\newline $\alpha=1$, $\beta=100$]{
\includegraphics[width=0.21\textwidth]{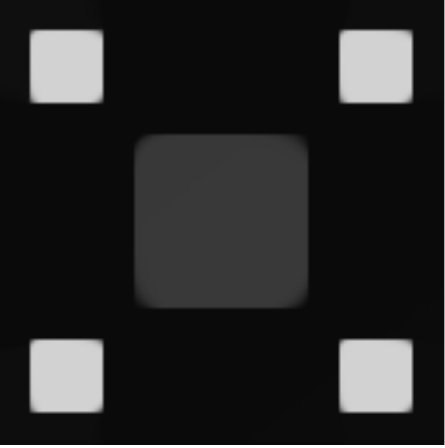}
}
\subfloat[$\tgv$ solution,\newline $\alpha=1$, $\beta=2$]{
\includegraphics[width=0.21\textwidth]{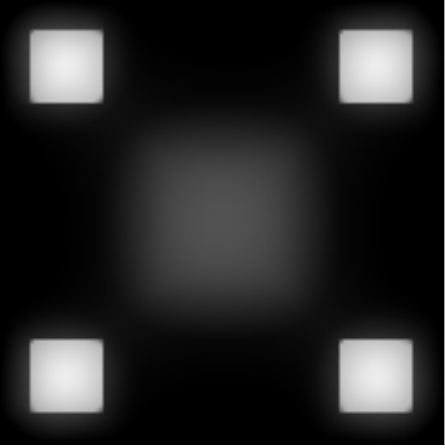}
}
\caption{Illustration of the two dimensional $\alpha\tv$ and $\tgvba$ equivalence for symmetric data when $\beta/\alpha$ is large enough.}
\label{squares}
\end{center}
\end{figure}

\begin{figure}[h!]
\begin{center}
\subfloat[Original image]{
\includegraphics[width=0.21\textwidth]{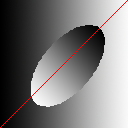}
}
\subfloat[Noisy image]{
\includegraphics[width=0.21\textwidth]{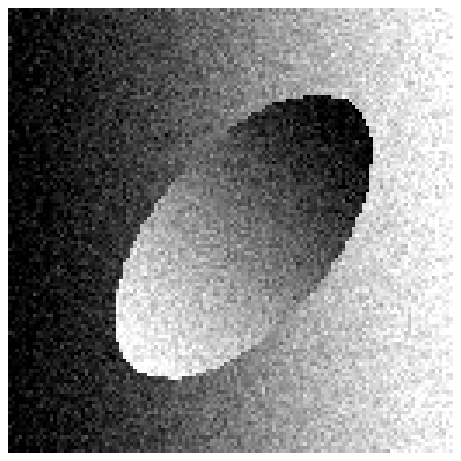}
}
\subfloat[$\tgv$ solution,\newline $\alpha=0.1$, $\beta=10^{-4}$]{
\includegraphics[width=0.21\textwidth]{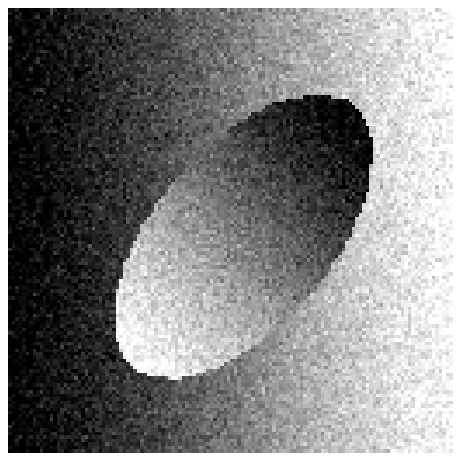}
}
\subfloat[$\tgv$ solution,\newline $\alpha=10^{-4}$, $\beta=0.15$]{
\includegraphics[width=0.21\textwidth]{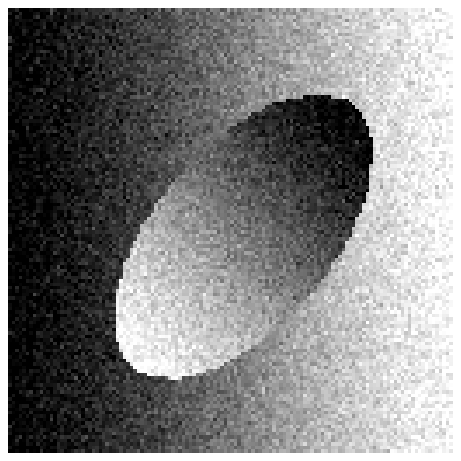}
}

\subfloat[$\tgv$ solution,\newline $\alpha=0.1$, $\beta=100$]{
\includegraphics[width=0.21\textwidth]{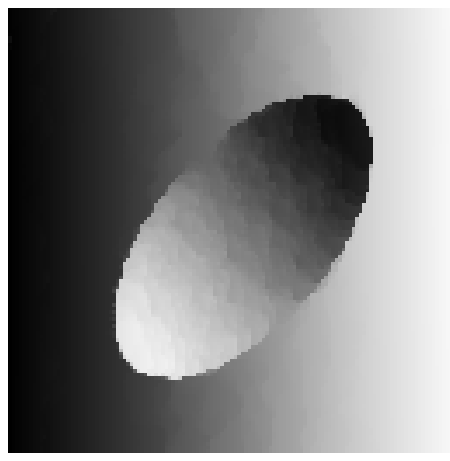}
}
\subfloat[$\tv$ solution,\newline $\alpha=0.1$]{
\includegraphics[width=0.21\textwidth]{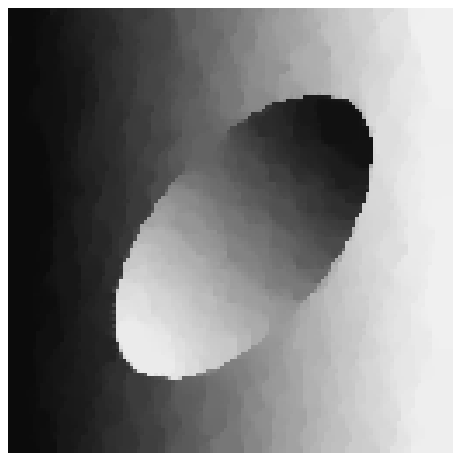}
}
\subfloat[$\tgv$ solution,\newline $\alpha=0.1$, $\beta=0.15$]{
\includegraphics[width=0.21\textwidth]{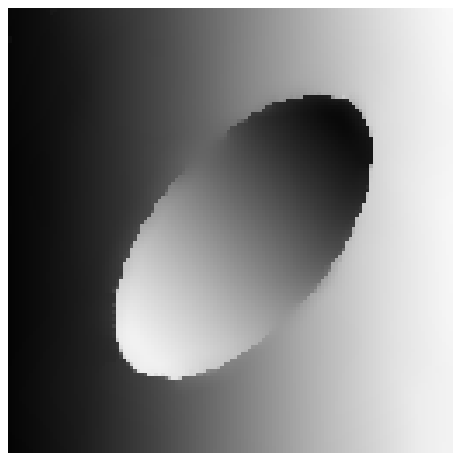}
}
\subfloat[$\tgv$ solution,\newline $\alpha=100$, $b=1000$]{
\includegraphics[width=0.21\textwidth]{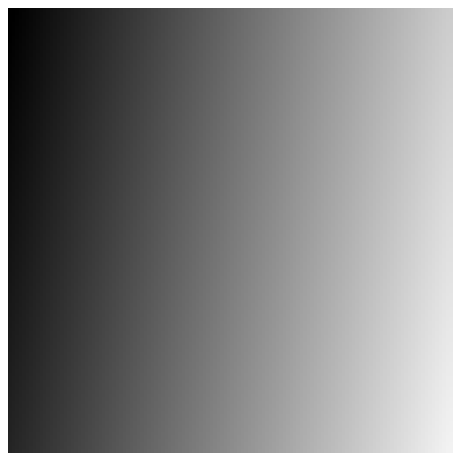}
}

\subfloat[Corresponding diagonal slices]{
 \resizebox{0.85\textwidth} {!}{\input{slices.tex}}
}

\caption{$L^{2}$--$\tgvba$ denoising for extremal values of $\alpha$ and $\beta$.}
\label{grad}
\end{center}
\end{figure}

Finally in Figure \ref{grad}, we solve the $L^{2}$--$\tgvba$ regularisation problem in a noisy image. We observe that for very small values of $\beta$ or $\alpha$, essentially we have no regularisation at all, see Figures \ref{grad}(c) and \ref{grad}(d) respectively, verifying Proposition \ref{lbl:todata}. In Figure \ref{grad}(e), we choose a large ratio $\beta/\alpha$, obtaining a $\tv$--like result which is nevertheless quite different than the $\alpha\tv$ result, Figure \ref{grad}(f), having staircasing only inside the ellipse.  This is due to the  ``affine'' correction predicted by Corollary \ref{lbl:TGV_largeba}, see also the corresponding diagonal slices in Figure \ref{grad}(i). Figure \ref{grad}(g) depicts a typical $\tgv$ solution with no staircasing while in Figure \ref{grad}(h) we set $\alpha$ and $\beta$ large enough and we obtain the linear regression of the data, as expected from Proposition \ref{lbl:regression_L2}.

\subsubsection*{Acknowledgements.} This work is supported by the King Abdullah University for Science and Technology (KAUST) Award No. KUK-I1-007-43. The first author acknowledges further support by the Cambridge Centre for Analysis (CCA) and the Engineering and Physical Sciences Research Council (EPSRC). The second author acknowledges further support from EPSRC grant EP/M00483X/1 ``Efficient computational tools for inverse imaging problems''.

\bibliographystyle{amsalpha}
\bibliography{kostasbib}

\end{document}

%% file: slices.tex
%
%
\definecolor{mycolor1}{rgb}{1.00000,0.00000,1.00000}%
\begin{tikzpicture}

\begin{axis}[%
width=21.850394in,
height=5.105362in,
at={(0.305212in,0.180352in)},
scale only axis,
separate axis lines,
every outer x axis line/.append style={black},
every x tick label/.append style={font=\color{black}},
xmin=-10,
xmax=660,
xticklabels={,,},
every outer y axis line/.append style={black},
every y tick label/.append style={font=\color{black}},
ymin=-0.2,
ymax=1.1,
ytick={0, 0.1, 0.2, 0.3, 0.4 ,0.5, 0.6 ,0.7,0.8, 0.9,1}
]
\addplot [color=black,solid,line width=2.0pt,forget plot]
  table[row sep=crcr]{%
1	0\\
2	0.00392156862745098\\
3	0.00784313725490196\\
4	0.0117647058823529\\
5	0.0179884042135338\\
6	0.0219099728409848\\
7	0.0274509803921569\\
8	0.0286232785148776\\
9	0.0352941176470588\\
10	0.0396628270755102\\
11	0.0466116827284114\\
12	0.0509803921568627\\
13	0.0572040904880436\\
14	0.0599958275344854\\
15	0.0666666666666667\\
16	0.0745098039215686\\
17	0.0811806430537499\\
18	0.0862745098039216\\
19	0.0918155173550936\\
20	0.0980392156862745\\
21	0.105882352941176\\
22	0.112106051272357\\
23	0.119949188527259\\
24	0.125490196078431\\
25	0.133333333333333\\
26	0.140004172465515\\
27	0.147400168919416\\
28	0.154560615394309\\
29	0.161956611848211\\
30	0.169074591781393\\
31	0.179945016061745\\
32	0.184313725490196\\
33	0.968627450980392\\
34	0.956862745098039\\
35	0.949019607843137\\
36	0.937254901960784\\
37	0.925490196078431\\
38	0.917647058823529\\
39	0.901960784313726\\
40	0.890196078431372\\
41	0.87843137254902\\
42	0.866666666666667\\
43	0.850980392156863\\
44	0.83921568627451\\
45	0.823529411764706\\
46	0.807843137254902\\
47	0.796078431372549\\
48	0.780392156862745\\
49	0.764705882352941\\
50	0.749019607843137\\
51	0.733333333333333\\
52	0.71764705882353\\
53	0.701960784313725\\
54	0.682352941176471\\
55	0.666666666666667\\
56	0.650980392156863\\
57	0.631372549019608\\
58	0.615686274509804\\
59	0.6\\
60	0.580392156862745\\
61	0.564705882352941\\
62	0.545098039215686\\
63	0.529411764705882\\
64	0.509803921568627\\
65	0.494117647058824\\
66	0.474509803921569\\
67	0.458823529411765\\
68	0.443137254901961\\
69	0.423529411764706\\
70	0.407843137254902\\
71	0.388235294117647\\
72	0.372549019607843\\
73	0.352941176470588\\
74	0.337254901960784\\
75	0.32156862745098\\
76	0.305882352941176\\
77	0.290196078431373\\
78	0.270588235294118\\
79	0.254901960784314\\
80	0.23921568627451\\
81	0.223529411764706\\
82	0.207843137254902\\
83	0.196078431372549\\
84	0.180392156862745\\
85	0.164705882352941\\
86	0.152941176470588\\
87	0.137254901960784\\
88	0.125490196078431\\
89	0.113725490196078\\
90	0.0980392156862745\\
91	0.0862745098039216\\
92	0.0745098039215686\\
93	0.0666666666666667\\
94	0.0549019607843137\\
95	0.0431372549019608\\
96	0.0352941176470588\\
97	0.0274509803921569\\
98	0.0156862745098039\\
99	0.00784313725490196\\
100	0\\
101	0.798827701877279\\
102	0.807843137254902\\
103	0.812211846683353\\
104	0.820054983938255\\
105	0.829070419315878\\
106	0.838043388151789\\
107	0.845886525406691\\
108	0.853729662661593\\
109	0.858823529411765\\
110	0.866666666666667\\
111	0.874509803921569\\
112	0.87960367067174\\
113	0.886274509803921\\
114	0.896419776762554\\
115	0.899658654609996\\
116	0.906329493742177\\
117	0.913725490196078\\
118	0.92112148664998\\
119	0.925490196078431\\
120	0.931031203629603\\
121	0.937254901960784\\
122	0.944650898414686\\
123	0.949019607843137\\
124	0.954560615394309\\
125	0.96033717292449\\
126	0.964705882352941\\
127	0.970929580684122\\
128	0.976470588235294\\
};
\addplot [color=mycolor1,solid,line width=2.0pt,forget plot]
  table[row sep=crcr]{%
131	0\\
132	0\\
133	0.0745098039215686\\
134	0\\
135	0\\
136	0.0274509803921569\\
137	0.0862745098039216\\
138	0.0901960784313726\\
139	0.0666666666666667\\
140	0\\
141	0.0627450980392157\\
142	0.0156862745098039\\
143	0.0274509803921569\\
144	0.0705882352941176\\
145	0.0196078431372549\\
146	0.0274509803921569\\
147	0.0509803921568627\\
148	0.152941176470588\\
149	0.0470588235294118\\
150	0.129411764705882\\
151	0.215686274509804\\
152	0.145098039215686\\
153	0.282352941176471\\
154	0.0941176470588235\\
155	0.105882352941176\\
156	0.0784313725490196\\
157	0.0666666666666667\\
158	0.282352941176471\\
159	0.0235294117647059\\
160	0.215686274509804\\
161	0.223529411764706\\
162	0.12156862745098\\
163	1\\
164	1\\
165	1\\
166	0.964705882352941\\
167	0.949019607843137\\
168	0.945098039215686\\
169	0.796078431372549\\
170	0.8\\
171	0.972549019607843\\
172	0.780392156862745\\
173	0.764705882352941\\
174	0.847058823529412\\
175	0.858823529411765\\
176	0.882352941176471\\
177	0.71764705882353\\
178	0.709803921568627\\
179	0.768627450980392\\
180	0.662745098039216\\
181	0.792156862745098\\
182	0.701960784313725\\
183	0.701960784313725\\
184	0.545098039215686\\
185	0.72156862745098\\
186	0.811764705882353\\
187	0.709803921568627\\
188	0.596078431372549\\
189	0.517647058823529\\
190	0.482352941176471\\
191	0.647058823529412\\
192	0.517647058823529\\
193	0.509803921568627\\
194	0.450980392156863\\
195	0.517647058823529\\
196	0.537254901960784\\
197	0.482352941176471\\
198	0.564705882352941\\
199	0.419607843137255\\
200	0.407843137254902\\
201	0.298039215686275\\
202	0.356862745098039\\
203	0.32156862745098\\
204	0.352941176470588\\
205	0.407843137254902\\
206	0.282352941176471\\
207	0.403921568627451\\
208	0.2\\
209	0.203921568627451\\
210	0.184313725490196\\
211	0.235294117647059\\
212	0.152941176470588\\
213	0.258823529411765\\
214	0.0941176470588235\\
215	0.0392156862745098\\
216	0.203921568627451\\
217	0.16078431372549\\
218	0.141176470588235\\
219	0.0980392156862745\\
220	0.168627450980392\\
221	0.0196078431372549\\
222	0.0470588235294118\\
223	0.211764705882353\\
224	0.0392156862745098\\
225	0.0392156862745098\\
226	0\\
227	0.00392156862745098\\
228	0\\
229	0.0705882352941176\\
230	0\\
231	0.67843137254902\\
232	0.784313725490196\\
233	0.854901960784314\\
234	0.788235294117647\\
235	0.956862745098039\\
236	0.933333333333333\\
237	0.729411764705882\\
238	0.788235294117647\\
239	0.796078431372549\\
240	0.815686274509804\\
241	0.83921568627451\\
242	0.819607843137255\\
243	0.862745098039216\\
244	0.784313725490196\\
245	1\\
246	0.996078431372549\\
247	1\\
248	0.870588235294118\\
249	0.937254901960784\\
250	1\\
251	0.886274509803921\\
252	0.976470588235294\\
253	0.874509803921569\\
254	1\\
255	0.831372549019608\\
256	1\\
257	0.937254901960784\\
258	1\\
};
\addplot [color=red,solid,line width=2.0pt,forget plot]
  table[row sep=crcr]{%
261	0\\
262	0.00784313725490196\\
263	0.0156862745098039\\
264	0.0196078431372549\\
265	0.0274509803921569\\
266	0.0313725490196078\\
267	0.0392156862745098\\
268	0.0470588235294118\\
269	0.0509803921568627\\
270	0.0549019607843137\\
271	0.0549019607843137\\
272	0.0627450980392157\\
273	0.0666666666666667\\
274	0.0745098039215686\\
275	0.0784313725490196\\
276	0.0862745098039216\\
277	0.0901960784313726\\
278	0.0980392156862745\\
279	0.105882352941176\\
280	0.109803921568627\\
281	0.117647058823529\\
282	0.12156862745098\\
283	0.129411764705882\\
284	0.137254901960784\\
285	0.141176470588235\\
286	0.149019607843137\\
287	0.156862745098039\\
288	0.16078431372549\\
289	0.168627450980392\\
290	0.172549019607843\\
291	0.184313725490196\\
292	0.223529411764706\\
293	0.901960784313726\\
294	0.909803921568627\\
295	0.917647058823529\\
296	0.92156862745098\\
297	0.929411764705882\\
298	0.917647058823529\\
299	0.882352941176471\\
300	0.882352941176471\\
301	0.890196078431372\\
302	0.827450980392157\\
303	0.831372549019608\\
304	0.83921568627451\\
305	0.835294117647059\\
306	0.827450980392157\\
307	0.796078431372549\\
308	0.776470588235294\\
309	0.764705882352941\\
310	0.745098039215686\\
311	0.741176470588235\\
312	0.725490196078431\\
313	0.713725490196078\\
314	0.682352941176471\\
315	0.67843137254902\\
316	0.686274509803922\\
317	0.647058823529412\\
318	0.615686274509804\\
319	0.592156862745098\\
320	0.576470588235294\\
321	0.576470588235294\\
322	0.549019607843137\\
323	0.517647058823529\\
324	0.486274509803922\\
325	0.47843137254902\\
326	0.466666666666667\\
327	0.466666666666667\\
328	0.470588235294118\\
329	0.415686274509804\\
330	0.396078431372549\\
331	0.380392156862745\\
332	0.368627450980392\\
333	0.345098039215686\\
334	0.341176470588235\\
335	0.337254901960784\\
336	0.305882352941176\\
337	0.309803921568627\\
338	0.262745098039216\\
339	0.250980392156863\\
340	0.254901960784314\\
341	0.247058823529412\\
342	0.180392156862745\\
343	0.184313725490196\\
344	0.156862745098039\\
345	0.141176470588235\\
346	0.145098039215686\\
347	0.133333333333333\\
348	0.133333333333333\\
349	0.117647058823529\\
350	0.117647058823529\\
351	0.0627450980392157\\
352	0.0549019607843137\\
353	0.0627450980392157\\
354	0.0549019607843137\\
355	0.0509803921568627\\
356	0.0509803921568627\\
357	0.0549019607843137\\
358	0.0627450980392157\\
359	0.0705882352941176\\
360	0.0980392156862745\\
361	0.733333333333333\\
362	0.796078431372549\\
363	0.815686274509804\\
364	0.823529411764706\\
365	0.827450980392157\\
366	0.835294117647059\\
367	0.843137254901961\\
368	0.847058823529412\\
369	0.854901960784314\\
370	0.858823529411765\\
371	0.866666666666667\\
372	0.874509803921569\\
373	0.87843137254902\\
374	0.886274509803921\\
375	0.894117647058824\\
376	0.898039215686275\\
377	0.905882352941176\\
378	0.909803921568627\\
379	0.917647058823529\\
380	0.925490196078431\\
381	0.933333333333333\\
382	0.937254901960784\\
383	0.945098039215686\\
384	0.949019607843137\\
385	0.956862745098039\\
386	0.964705882352941\\
387	0.968627450980392\\
388	0.972549019607843\\
};
\addplot [color=blue,solid,line width=2.0pt,forget plot]
  table[row sep=crcr]{%
391	0.0431372549019608\\
392	0.0431372549019608\\
393	0.0431372549019608\\
394	0.0431372549019608\\
395	0.0431372549019608\\
396	0.0431372549019608\\
397	0.0431372549019608\\
398	0.0509803921568627\\
399	0.0509803921568627\\
400	0.0509803921568627\\
401	0.0509803921568627\\
402	0.0588235294117647\\
403	0.0705882352941176\\
404	0.0784313725490196\\
405	0.0784313725490196\\
406	0.0784313725490196\\
407	0.0901960784313726\\
408	0.101960784313725\\
409	0.101960784313725\\
410	0.105882352941176\\
411	0.105882352941176\\
412	0.105882352941176\\
413	0.12156862745098\\
414	0.12156862745098\\
415	0.12156862745098\\
416	0.125490196078431\\
417	0.137254901960784\\
418	0.137254901960784\\
419	0.137254901960784\\
420	0.141176470588235\\
421	0.16078431372549\\
422	0.223529411764706\\
423	0.92156862745098\\
424	0.92156862745098\\
425	0.92156862745098\\
426	0.92156862745098\\
427	0.92156862745098\\
428	0.913725490196078\\
429	0.882352941176471\\
430	0.882352941176471\\
431	0.882352941176471\\
432	0.83921568627451\\
433	0.835294117647059\\
434	0.835294117647059\\
435	0.831372549019608\\
436	0.823529411764706\\
437	0.8\\
438	0.776470588235294\\
439	0.764705882352941\\
440	0.749019607843137\\
441	0.741176470588235\\
442	0.725490196078431\\
443	0.709803921568627\\
444	0.682352941176471\\
445	0.682352941176471\\
446	0.682352941176471\\
447	0.647058823529412\\
448	0.615686274509804\\
449	0.596078431372549\\
450	0.576470588235294\\
451	0.572549019607843\\
452	0.549019607843137\\
453	0.521568627450981\\
454	0.490196078431373\\
455	0.474509803921569\\
456	0.466666666666667\\
457	0.462745098039216\\
458	0.462745098039216\\
459	0.415686274509804\\
460	0.403921568627451\\
461	0.380392156862745\\
462	0.368627450980392\\
463	0.345098039215686\\
464	0.341176470588235\\
465	0.333333333333333\\
466	0.305882352941176\\
467	0.305882352941176\\
468	0.270588235294118\\
469	0.250980392156863\\
470	0.247058823529412\\
471	0.23921568627451\\
472	0.188235294117647\\
473	0.180392156862745\\
474	0.16078431372549\\
475	0.149019607843137\\
476	0.145098039215686\\
477	0.137254901960784\\
478	0.129411764705882\\
479	0.117647058823529\\
480	0.113725490196078\\
481	0.0745098039215686\\
482	0.0666666666666667\\
483	0.0666666666666667\\
484	0.0666666666666667\\
485	0.0666666666666667\\
486	0.0666666666666667\\
487	0.0666666666666667\\
488	0.0666666666666667\\
489	0.0666666666666667\\
490	0.0941176470588235\\
491	0.733333333333333\\
492	0.792156862745098\\
493	0.823529411764706\\
494	0.831372549019608\\
495	0.847058823529412\\
496	0.847058823529412\\
497	0.843137254901961\\
498	0.843137254901961\\
499	0.847058823529412\\
500	0.850980392156863\\
501	0.850980392156863\\
502	0.870588235294118\\
503	0.870588235294118\\
504	0.870588235294118\\
505	0.901960784313726\\
506	0.905882352941176\\
507	0.917647058823529\\
508	0.92156862745098\\
509	0.941176470588235\\
510	0.941176470588235\\
511	0.941176470588235\\
512	0.941176470588235\\
513	0.941176470588235\\
514	0.941176470588235\\
515	0.941176470588235\\
516	0.941176470588235\\
517	0.941176470588235\\
518	0.941176470588235\\
};
\addplot [color=green,solid,line width=2.0pt,forget plot]
  table[row sep=crcr]{%
521	0.0156862745098039\\
522	0.0196078431372549\\
523	0.0235294117647059\\
524	0.0274509803921569\\
525	0.0352941176470588\\
526	0.0392156862745098\\
527	0.0431372549019608\\
528	0.0470588235294118\\
529	0.0509803921568627\\
530	0.0588235294117647\\
531	0.0627450980392157\\
532	0.0666666666666667\\
533	0.0705882352941176\\
534	0.0745098039215686\\
535	0.0823529411764706\\
536	0.0862745098039216\\
537	0.0901960784313726\\
538	0.0941176470588235\\
539	0.0980392156862745\\
540	0.105882352941176\\
541	0.109803921568627\\
542	0.113725490196078\\
543	0.12156862745098\\
544	0.125490196078431\\
545	0.133333333333333\\
546	0.137254901960784\\
547	0.145098039215686\\
548	0.149019607843137\\
549	0.156862745098039\\
550	0.164705882352941\\
551	0.184313725490196\\
552	0.223529411764706\\
553	0.92156862745098\\
554	0.933333333333333\\
555	0.941176470588235\\
556	0.937254901960784\\
557	0.929411764705882\\
558	0.92156862745098\\
559	0.905882352941176\\
560	0.894117647058824\\
561	0.882352941176471\\
562	0.866666666666667\\
563	0.854901960784314\\
564	0.843137254901961\\
565	0.827450980392157\\
566	0.811764705882353\\
567	0.8\\
568	0.784313725490196\\
569	0.768627450980392\\
570	0.752941176470588\\
571	0.733333333333333\\
572	0.71764705882353\\
573	0.701960784313725\\
574	0.686274509803922\\
575	0.666666666666667\\
576	0.650980392156863\\
577	0.631372549019608\\
578	0.615686274509804\\
579	0.596078431372549\\
580	0.580392156862745\\
581	0.56078431372549\\
582	0.545098039215686\\
583	0.525490196078431\\
584	0.509803921568627\\
585	0.490196078431373\\
586	0.474509803921569\\
587	0.458823529411765\\
588	0.43921568627451\\
589	0.423529411764706\\
590	0.403921568627451\\
591	0.388235294117647\\
592	0.372549019607843\\
593	0.356862745098039\\
594	0.337254901960784\\
595	0.32156862745098\\
596	0.305882352941176\\
597	0.286274509803922\\
598	0.270588235294118\\
599	0.254901960784314\\
600	0.23921568627451\\
601	0.223529411764706\\
602	0.203921568627451\\
603	0.188235294117647\\
604	0.172549019607843\\
605	0.16078431372549\\
606	0.145098039215686\\
607	0.133333333333333\\
608	0.12156862745098\\
609	0.109803921568627\\
610	0.0941176470588235\\
611	0.0823529411764706\\
612	0.0705882352941176\\
613	0.0588235294117647\\
614	0.0470588235294118\\
615	0.0392156862745098\\
616	0.0352941176470588\\
617	0.0392156862745098\\
618	0.0549019607843137\\
619	0.0745098039215686\\
620	0.109803921568627\\
621	0.729411764705882\\
622	0.780392156862745\\
623	0.803921568627451\\
624	0.819607843137255\\
625	0.827450980392157\\
626	0.835294117647059\\
627	0.843137254901961\\
628	0.847058823529412\\
629	0.854901960784314\\
630	0.862745098039216\\
631	0.870588235294118\\
632	0.87843137254902\\
633	0.882352941176471\\
634	0.890196078431372\\
635	0.898039215686275\\
636	0.905882352941176\\
637	0.913725490196078\\
638	0.92156862745098\\
639	0.925490196078431\\
640	0.933333333333333\\
641	0.937254901960784\\
642	0.945098039215686\\
643	0.949019607843137\\
644	0.956862745098039\\
645	0.96078431372549\\
646	0.964705882352941\\
647	0.972549019607843\\
648	0.972549019607843\\
};
\addplot [color=black,solid,line width=2.0pt]
  table[row sep=crcr]{%
1	0\\
2	0.00392156862745098\\
3	0.00784313725490196\\
4	0.0117647058823529\\
5	0.0179884042135338\\
6	0.0219099728409848\\
7	0.0274509803921569\\
8	0.0286232785148776\\
9	0.0352941176470588\\
10	0.0396628270755102\\
11	0.0466116827284114\\
12	0.0509803921568627\\
13	0.0572040904880436\\
14	0.0599958275344854\\
15	0.0666666666666667\\
16	0.0745098039215686\\
17	0.0811806430537499\\
18	0.0862745098039216\\
19	0.0918155173550936\\
20	0.0980392156862745\\
21	0.105882352941176\\
22	0.112106051272357\\
23	0.119949188527259\\
24	0.125490196078431\\
25	0.133333333333333\\
26	0.140004172465515\\
27	0.147400168919416\\
28	0.154560615394309\\
29	0.161956611848211\\
30	0.169074591781393\\
31	0.179945016061745\\
32	0.184313725490196\\
33	0.968627450980392\\
34	0.956862745098039\\
35	0.949019607843137\\
36	0.937254901960784\\
37	0.925490196078431\\
38	0.917647058823529\\
39	0.901960784313726\\
40	0.890196078431372\\
41	0.87843137254902\\
42	0.866666666666667\\
43	0.850980392156863\\
44	0.83921568627451\\
45	0.823529411764706\\
46	0.807843137254902\\
47	0.796078431372549\\
48	0.780392156862745\\
49	0.764705882352941\\
50	0.749019607843137\\
51	0.733333333333333\\
52	0.71764705882353\\
53	0.701960784313725\\
54	0.682352941176471\\
55	0.666666666666667\\
56	0.650980392156863\\
57	0.631372549019608\\
58	0.615686274509804\\
59	0.6\\
60	0.580392156862745\\
61	0.564705882352941\\
62	0.545098039215686\\
63	0.529411764705882\\
64	0.509803921568627\\
65	0.494117647058824\\
66	0.474509803921569\\
67	0.458823529411765\\
68	0.443137254901961\\
69	0.423529411764706\\
70	0.407843137254902\\
71	0.388235294117647\\
72	0.372549019607843\\
73	0.352941176470588\\
74	0.337254901960784\\
75	0.32156862745098\\
76	0.305882352941176\\
77	0.290196078431373\\
78	0.270588235294118\\
79	0.254901960784314\\
80	0.23921568627451\\
81	0.223529411764706\\
82	0.207843137254902\\
83	0.196078431372549\\
84	0.180392156862745\\
85	0.164705882352941\\
86	0.152941176470588\\
87	0.137254901960784\\
88	0.125490196078431\\
89	0.113725490196078\\
90	0.0980392156862745\\
91	0.0862745098039216\\
92	0.0745098039215686\\
93	0.0666666666666667\\
94	0.0549019607843137\\
95	0.0431372549019608\\
96	0.0352941176470588\\
97	0.0274509803921569\\
98	0.0156862745098039\\
99	0.00784313725490196\\
100	0\\
101	0.798827701877279\\
102	0.807843137254902\\
103	0.812211846683353\\
104	0.820054983938255\\
105	0.829070419315878\\
106	0.838043388151789\\
107	0.845886525406691\\
108	0.853729662661593\\
109	0.858823529411765\\
110	0.866666666666667\\
111	0.874509803921569\\
112	0.87960367067174\\
113	0.886274509803921\\
114	0.896419776762554\\
115	0.899658654609996\\
116	0.906329493742177\\
117	0.913725490196078\\
118	0.92112148664998\\
119	0.925490196078431\\
120	0.931031203629603\\
121	0.937254901960784\\
122	0.944650898414686\\
123	0.949019607843137\\
124	0.954560615394309\\
125	0.96033717292449\\
126	0.964705882352941\\
127	0.970929580684122\\
128	0.976470588235294\\
};

\addplot [color=mycolor1,solid,line width=2.0pt]
  table[row sep=crcr]{%
131	0\\
132	0\\
133	0.0745098039215686\\
134	0\\
135	0\\
136	0.0274509803921569\\
137	0.0862745098039216\\
138	0.0901960784313726\\
139	0.0666666666666667\\
140	0\\
141	0.0627450980392157\\
142	0.0156862745098039\\
143	0.0274509803921569\\
144	0.0705882352941176\\
145	0.0196078431372549\\
146	0.0274509803921569\\
147	0.0509803921568627\\
148	0.152941176470588\\
149	0.0470588235294118\\
150	0.129411764705882\\
151	0.215686274509804\\
152	0.145098039215686\\
153	0.282352941176471\\
154	0.0941176470588235\\
155	0.105882352941176\\
156	0.0784313725490196\\
157	0.0666666666666667\\
158	0.282352941176471\\
159	0.0235294117647059\\
160	0.215686274509804\\
161	0.223529411764706\\
162	0.12156862745098\\
163	1\\
164	1\\
165	1\\
166	0.964705882352941\\
167	0.949019607843137\\
168	0.945098039215686\\
169	0.796078431372549\\
170	0.8\\
171	0.972549019607843\\
172	0.780392156862745\\
173	0.764705882352941\\
174	0.847058823529412\\
175	0.858823529411765\\
176	0.882352941176471\\
177	0.71764705882353\\
178	0.709803921568627\\
179	0.768627450980392\\
180	0.662745098039216\\
181	0.792156862745098\\
182	0.701960784313725\\
183	0.701960784313725\\
184	0.545098039215686\\
185	0.72156862745098\\
186	0.811764705882353\\
187	0.709803921568627\\
188	0.596078431372549\\
189	0.517647058823529\\
190	0.482352941176471\\
191	0.647058823529412\\
192	0.517647058823529\\
193	0.509803921568627\\
194	0.450980392156863\\
195	0.517647058823529\\
196	0.537254901960784\\
197	0.482352941176471\\
198	0.564705882352941\\
199	0.419607843137255\\
200	0.407843137254902\\
201	0.298039215686275\\
202	0.356862745098039\\
203	0.32156862745098\\
204	0.352941176470588\\
205	0.407843137254902\\
206	0.282352941176471\\
207	0.403921568627451\\
208	0.2\\
209	0.203921568627451\\
210	0.184313725490196\\
211	0.235294117647059\\
212	0.152941176470588\\
213	0.258823529411765\\
214	0.0941176470588235\\
215	0.0392156862745098\\
216	0.203921568627451\\
217	0.16078431372549\\
218	0.141176470588235\\
219	0.0980392156862745\\
220	0.168627450980392\\
221	0.0196078431372549\\
222	0.0470588235294118\\
223	0.211764705882353\\
224	0.0392156862745098\\
225	0.0392156862745098\\
226	0\\
227	0.00392156862745098\\
228	0\\
229	0.0705882352941176\\
230	0\\
231	0.67843137254902\\
232	0.784313725490196\\
233	0.854901960784314\\
234	0.788235294117647\\
235	0.956862745098039\\
236	0.933333333333333\\
237	0.729411764705882\\
238	0.788235294117647\\
239	0.796078431372549\\
240	0.815686274509804\\
241	0.83921568627451\\
242	0.819607843137255\\
243	0.862745098039216\\
244	0.784313725490196\\
245	1\\
246	0.996078431372549\\
247	1\\
248	0.870588235294118\\
249	0.937254901960784\\
250	1\\
251	0.886274509803921\\
252	0.976470588235294\\
253	0.874509803921569\\
254	1\\
255	0.831372549019608\\
256	1\\
257	0.937254901960784\\
258	1\\
};

\addplot [color=red,solid,line width=2.0pt]
  table[row sep=crcr]{%
261	0\\
262	0.00784313725490196\\
263	0.0156862745098039\\
264	0.0196078431372549\\
265	0.0274509803921569\\
266	0.0313725490196078\\
267	0.0392156862745098\\
268	0.0470588235294118\\
269	0.0509803921568627\\
270	0.0549019607843137\\
271	0.0549019607843137\\
272	0.0627450980392157\\
273	0.0666666666666667\\
274	0.0745098039215686\\
275	0.0784313725490196\\
276	0.0862745098039216\\
277	0.0901960784313726\\
278	0.0980392156862745\\
279	0.105882352941176\\
280	0.109803921568627\\
281	0.117647058823529\\
282	0.12156862745098\\
283	0.129411764705882\\
284	0.137254901960784\\
285	0.141176470588235\\
286	0.149019607843137\\
287	0.156862745098039\\
288	0.16078431372549\\
289	0.168627450980392\\
290	0.172549019607843\\
291	0.184313725490196\\
292	0.223529411764706\\
293	0.901960784313726\\
294	0.909803921568627\\
295	0.917647058823529\\
296	0.92156862745098\\
297	0.929411764705882\\
298	0.917647058823529\\
299	0.882352941176471\\
300	0.882352941176471\\
301	0.890196078431372\\
302	0.827450980392157\\
303	0.831372549019608\\
304	0.83921568627451\\
305	0.835294117647059\\
306	0.827450980392157\\
307	0.796078431372549\\
308	0.776470588235294\\
309	0.764705882352941\\
310	0.745098039215686\\
311	0.741176470588235\\
312	0.725490196078431\\
313	0.713725490196078\\
314	0.682352941176471\\
315	0.67843137254902\\
316	0.686274509803922\\
317	0.647058823529412\\
318	0.615686274509804\\
319	0.592156862745098\\
320	0.576470588235294\\
321	0.576470588235294\\
322	0.549019607843137\\
323	0.517647058823529\\
324	0.486274509803922\\
325	0.47843137254902\\
326	0.466666666666667\\
327	0.466666666666667\\
328	0.470588235294118\\
329	0.415686274509804\\
330	0.396078431372549\\
331	0.380392156862745\\
332	0.368627450980392\\
333	0.345098039215686\\
334	0.341176470588235\\
335	0.337254901960784\\
336	0.305882352941176\\
337	0.309803921568627\\
338	0.262745098039216\\
339	0.250980392156863\\
340	0.254901960784314\\
341	0.247058823529412\\
342	0.180392156862745\\
343	0.184313725490196\\
344	0.156862745098039\\
345	0.141176470588235\\
346	0.145098039215686\\
347	0.133333333333333\\
348	0.133333333333333\\
349	0.117647058823529\\
350	0.117647058823529\\
351	0.0627450980392157\\
352	0.0549019607843137\\
353	0.0627450980392157\\
354	0.0549019607843137\\
355	0.0509803921568627\\
356	0.0509803921568627\\
357	0.0549019607843137\\
358	0.0627450980392157\\
359	0.0705882352941176\\
360	0.0980392156862745\\
361	0.733333333333333\\
362	0.796078431372549\\
363	0.815686274509804\\
364	0.823529411764706\\
365	0.827450980392157\\
366	0.835294117647059\\
367	0.843137254901961\\
368	0.847058823529412\\
369	0.854901960784314\\
370	0.858823529411765\\
371	0.866666666666667\\
372	0.874509803921569\\
373	0.87843137254902\\
374	0.886274509803921\\
375	0.894117647058824\\
376	0.898039215686275\\
377	0.905882352941176\\
378	0.909803921568627\\
379	0.917647058823529\\
380	0.925490196078431\\
381	0.933333333333333\\
382	0.937254901960784\\
383	0.945098039215686\\
384	0.949019607843137\\
385	0.956862745098039\\
386	0.964705882352941\\
387	0.968627450980392\\
388	0.972549019607843\\
};

\addplot [color=blue,solid,line width=2.0pt]
  table[row sep=crcr]{%
391	0.0431372549019608\\
392	0.0431372549019608\\
393	0.0431372549019608\\
394	0.0431372549019608\\
395	0.0431372549019608\\
396	0.0431372549019608\\
397	0.0431372549019608\\
398	0.0509803921568627\\
399	0.0509803921568627\\
400	0.0509803921568627\\
401	0.0509803921568627\\
402	0.0588235294117647\\
403	0.0705882352941176\\
404	0.0784313725490196\\
405	0.0784313725490196\\
406	0.0784313725490196\\
407	0.0901960784313726\\
408	0.101960784313725\\
409	0.101960784313725\\
410	0.105882352941176\\
411	0.105882352941176\\
412	0.105882352941176\\
413	0.12156862745098\\
414	0.12156862745098\\
415	0.12156862745098\\
416	0.125490196078431\\
417	0.137254901960784\\
418	0.137254901960784\\
419	0.137254901960784\\
420	0.141176470588235\\
421	0.16078431372549\\
422	0.223529411764706\\
423	0.92156862745098\\
424	0.92156862745098\\
425	0.92156862745098\\
426	0.92156862745098\\
427	0.92156862745098\\
428	0.913725490196078\\
429	0.882352941176471\\
430	0.882352941176471\\
431	0.882352941176471\\
432	0.83921568627451\\
433	0.835294117647059\\
434	0.835294117647059\\
435	0.831372549019608\\
436	0.823529411764706\\
437	0.8\\
438	0.776470588235294\\
439	0.764705882352941\\
440	0.749019607843137\\
441	0.741176470588235\\
442	0.725490196078431\\
443	0.709803921568627\\
444	0.682352941176471\\
445	0.682352941176471\\
446	0.682352941176471\\
447	0.647058823529412\\
448	0.615686274509804\\
449	0.596078431372549\\
450	0.576470588235294\\
451	0.572549019607843\\
452	0.549019607843137\\
453	0.521568627450981\\
454	0.490196078431373\\
455	0.474509803921569\\
456	0.466666666666667\\
457	0.462745098039216\\
458	0.462745098039216\\
459	0.415686274509804\\
460	0.403921568627451\\
461	0.380392156862745\\
462	0.368627450980392\\
463	0.345098039215686\\
464	0.341176470588235\\
465	0.333333333333333\\
466	0.305882352941176\\
467	0.305882352941176\\
468	0.270588235294118\\
469	0.250980392156863\\
470	0.247058823529412\\
471	0.23921568627451\\
472	0.188235294117647\\
473	0.180392156862745\\
474	0.16078431372549\\
475	0.149019607843137\\
476	0.145098039215686\\
477	0.137254901960784\\
478	0.129411764705882\\
479	0.117647058823529\\
480	0.113725490196078\\
481	0.0745098039215686\\
482	0.0666666666666667\\
483	0.0666666666666667\\
484	0.0666666666666667\\
485	0.0666666666666667\\
486	0.0666666666666667\\
487	0.0666666666666667\\
488	0.0666666666666667\\
489	0.0666666666666667\\
490	0.0941176470588235\\
491	0.733333333333333\\
492	0.792156862745098\\
493	0.823529411764706\\
494	0.831372549019608\\
495	0.847058823529412\\
496	0.847058823529412\\
497	0.843137254901961\\
498	0.843137254901961\\
499	0.847058823529412\\
500	0.850980392156863\\
501	0.850980392156863\\
502	0.870588235294118\\
503	0.870588235294118\\
504	0.870588235294118\\
505	0.901960784313726\\
506	0.905882352941176\\
507	0.917647058823529\\
508	0.92156862745098\\
509	0.941176470588235\\
510	0.941176470588235\\
511	0.941176470588235\\
512	0.941176470588235\\
513	0.941176470588235\\
514	0.941176470588235\\
515	0.941176470588235\\
516	0.941176470588235\\
517	0.941176470588235\\
518	0.941176470588235\\
};

\addplot [color=green,solid,line width=2.0pt]
  table[row sep=crcr]{%
521	0.0156862745098039\\
522	0.0196078431372549\\
523	0.0235294117647059\\
524	0.0274509803921569\\
525	0.0352941176470588\\
526	0.0392156862745098\\
527	0.0431372549019608\\
528	0.0470588235294118\\
529	0.0509803921568627\\
530	0.0588235294117647\\
531	0.0627450980392157\\
532	0.0666666666666667\\
533	0.0705882352941176\\
534	0.0745098039215686\\
535	0.0823529411764706\\
536	0.0862745098039216\\
537	0.0901960784313726\\
538	0.0941176470588235\\
539	0.0980392156862745\\
540	0.105882352941176\\
541	0.109803921568627\\
542	0.113725490196078\\
543	0.12156862745098\\
544	0.125490196078431\\
545	0.133333333333333\\
546	0.137254901960784\\
547	0.145098039215686\\
548	0.149019607843137\\
549	0.156862745098039\\
550	0.164705882352941\\
551	0.184313725490196\\
552	0.223529411764706\\
553	0.92156862745098\\
554	0.933333333333333\\
555	0.941176470588235\\
556	0.937254901960784\\
557	0.929411764705882\\
558	0.92156862745098\\
559	0.905882352941176\\
560	0.894117647058824\\
561	0.882352941176471\\
562	0.866666666666667\\
563	0.854901960784314\\
564	0.843137254901961\\
565	0.827450980392157\\
566	0.811764705882353\\
567	0.8\\
568	0.784313725490196\\
569	0.768627450980392\\
570	0.752941176470588\\
571	0.733333333333333\\
572	0.71764705882353\\
573	0.701960784313725\\
574	0.686274509803922\\
575	0.666666666666667\\
576	0.650980392156863\\
577	0.631372549019608\\
578	0.615686274509804\\
579	0.596078431372549\\
580	0.580392156862745\\
581	0.56078431372549\\
582	0.545098039215686\\
583	0.525490196078431\\
584	0.509803921568627\\
585	0.490196078431373\\
586	0.474509803921569\\
587	0.458823529411765\\
588	0.43921568627451\\
589	0.423529411764706\\
590	0.403921568627451\\
591	0.388235294117647\\
592	0.372549019607843\\
593	0.356862745098039\\
594	0.337254901960784\\
595	0.32156862745098\\
596	0.305882352941176\\
597	0.286274509803922\\
598	0.270588235294118\\
599	0.254901960784314\\
600	0.23921568627451\\
601	0.223529411764706\\
602	0.203921568627451\\
603	0.188235294117647\\
604	0.172549019607843\\
605	0.16078431372549\\
606	0.145098039215686\\
607	0.133333333333333\\
608	0.12156862745098\\
609	0.109803921568627\\
610	0.0941176470588235\\
611	0.0823529411764706\\
612	0.0705882352941176\\
613	0.0588235294117647\\
614	0.0470588235294118\\
615	0.0392156862745098\\
616	0.0352941176470588\\
617	0.0392156862745098\\
618	0.0549019607843137\\
619	0.0745098039215686\\
620	0.109803921568627\\
621	0.729411764705882\\
622	0.780392156862745\\
623	0.803921568627451\\
624	0.819607843137255\\
625	0.827450980392157\\
626	0.835294117647059\\
627	0.843137254901961\\
628	0.847058823529412\\
629	0.854901960784314\\
630	0.862745098039216\\
631	0.870588235294118\\
632	0.87843137254902\\
633	0.882352941176471\\
634	0.890196078431372\\
635	0.898039215686275\\
636	0.905882352941176\\
637	0.913725490196078\\
638	0.92156862745098\\
639	0.925490196078431\\
640	0.933333333333333\\
641	0.937254901960784\\
642	0.945098039215686\\
643	0.949019607843137\\
644	0.956862745098039\\
645	0.96078431372549\\
646	0.964705882352941\\
647	0.972549019607843\\
648	0.972549019607843\\
};
\node[align=center, inner sep=0mm, text=black]
at (axis cs:60,-0.1,0) {\Huge{Original}};
\node[align=center, inner sep=0mm, text=black]
at (axis cs:190,-0.1,0) {\Huge{Noisy}};
\node[align=center, inner sep=0mm, text=black]
at (axis cs:320,-0.1,0) {\Huge{TGV: $\alpha=0.1$, $\beta=100$}};
\node[align=center, inner sep=0mm, text=black]
at (axis cs:450,-0.1,0) {\Huge{TV: $\alpha=0.1$}};
\node[align=center, inner sep=0mm, text=black]
at (axis cs:580,-0.1,0) {\Huge{TGV: $\alpha=0.1$, $\beta=0.15$}};
\end{axis}
\end{tikzpicture}%